\documentclass[10pt,a4paper]{amsart}

\usepackage[a4paper,top=3.5 cm,bottom=3 cm,left=2.5 cm,right=2.5 cm]{geometry}
\allowdisplaybreaks
\usepackage[T1]{fontenc}
\usepackage[english]{babel}
\usepackage{url}
\usepackage{amsmath,amssymb,amsthm}
\usepackage{mathtools}
\usepackage{graphicx}
\usepackage{xcolor} 
\usepackage{bbm}
\usepackage{dsfont}
\usepackage{stmaryrd}
\usepackage{esint}
\usepackage{floatrow}
\usepackage{epsfig}
\usepackage{hyperref}
\usepackage[normalem]{ulem}
\usepackage{mathrsfs}
\usepackage{tikz}
\usetikzlibrary{decorations.markings,backgrounds}
\usetikzlibrary{arrows.meta}
\usepackage{pgfplots}\pgfplotsset{compat=1.18}
\usepackage{pst-plot}
\usepackage{subfig}
\usepackage{caption}
\usepackage{float}
\usepackage[colorinlistoftodos]{todonotes}
\usepackage{import}
\usepackage{enumitem}
\usepackage{verbatim,fancyvrb}
\usepackage{epigraph}

\usepackage[nameinlink,capitalise,sort]{cleveref}
\crefname{equation}{}{} % Makes cref behave like eqref.
\crefname{enumi}{}{} % Avoid showing ``item (1)''

\definecolor{mblue}{rgb}{         0    0.4470    0.7410}
\definecolor{mred}{rgb}{    0.8500    0.3250    0.0980}  
\definecolor{myellow}{rgb}{    0.9290    0.6940    0.1250}

\theoremstyle{plain}
\newtheorem{lemma}{Lemma}[section]
\newtheorem{proposition}[lemma]{Proposition}
\newtheorem{theorem}[lemma]{Theorem}

\newtheorem*{idea*}{Idea}

\theoremstyle{definition}
\newtheorem{definition}[lemma]{Definition}

\theoremstyle{remark}
\newtheorem{remark}[lemma]{Remark}

\numberwithin{equation}{section}

\newcommand{\loc}{\mathrm{loc}}

\newcommand{\R}{\mathbb{R}}
\newcommand{\N}{\mathbb{N}}

\newcommand{\eps}{{\varepsilon}}

\DeclareMathOperator{\supp}{supp}
\DeclareMathOperator{\dive}{div}

\DeclareMathOperator*{\esssup}{ess\,\sup}

\newcommand{\wstarlim}{\overset{\ast}\rightharpoonup}
\newcommand{\wlim}{\rightharpoonup}

\newcommand{\pt}{\partial_t}

\newcommand{\abs}[1]{\left\lvert#1\right\rvert}
\newcommand{\norm}[1]{\left\lVert#1\right\rVert}
\newcommand\skal[1]{\left\langle#1\right\rangle}

\renewcommand{\div}{\operatorname{div}}

\newcommand{\dd}{\,\mathrm{d}}

\newcommand{\ie}{{\itshape i.e.}}

\newcounter{eq}[section]

\title[Inhomogeneous 2D Navier--Stokes equations without vacuum]{Inhomogeneous 2D Navier--Stokes equations: Existence, uniqueness, stability, continuity in time and energy conservation of weak solutions}

\date{}

\author[S. \v{S}kondri\'c]{Stefan \v{S}kondri\'c}
\address[S. \v{S}kondri\'c]{Chair of Analysis, Department of
Mathematics, Friedrich-Alexander-Universität Erlangen-Nürnberg, Cauerstraße 11, 91058 Erlangen, Germany.}
\email{stefan.skondric@fau.de}

\begin{document}

\begin{abstract}
We present a novel and direct proof of the existence and uniqueness of weak solutions of the inhomogeneous incompressible Navier--Stokes equations without vacuum.  The analysis we employ to prove the strong convergence of the approximating sequence, which is based on the relative energy method, reveals how to conclude the stability and uniqueness of weak solutions. To the best of our knowledge, these global-in-time stability estimates are completely new. Furthermore, for the first time, we establish energy conservation for weak solutions. 
\end{abstract}

\maketitle

\section{Introduction}\label{sec:section1} 

\subsection{Notion of solution and main result}

In this paper, we consider the \emph{inhomogeneous incompressible Navier--Stokes equations}
\begin{align}\label{eq:ins}
\begin{cases}
\pt (\rho u) +\div (\rho u \otimes u)-\nu \Delta u+\nabla P=0, &  (t,x) \in (0,\infty) \times \R^d, \\
\pt \rho +\div (\rho u )=0, & (t,x) \in (0,\infty) \times \R^d, \\
\operatorname{div} u=0, & (t,x) \in (0,\infty) \times \R^d, \\
\rho(0,x) = \rho_0(x), & x \in \R^d,\\
u(0,x) = u_0(x), & x \in \R^d.
\end{cases}
\end{align} 
System \cref{eq:ins} describes the evolution of an incompressible fluid with variable \emph{density} $\rho \colon [0,\infty) \times \R^d \to [0,\infty)$, \emph{velocity} $u \colon [0,\infty) \times \R^d \to \R^d$ and \emph{pressure} $P \colon [0,\infty) \times \R^d \to \R$. The constant $\nu>0$ denotes the viscosity of the fluid. 

Recall that a sufficiently smooth and integrable solution $(\rho,u,P)$ of \cref{eq:ins}, with initial data $(\rho_0,u_0)$, satisfies the energy equality, \ie, for every $t>0$, we have 
\begin{align}\label{eq:energyeq}
    \frac{1}{2} \int_{\R^d} \rho(t) |u(t)|^2 \dd x + \nu \int_0^t \int_{\R^d} |\nabla u(s)|^2 \dd x \dd s = \frac{1}{2} \int_{\R^d} \rho_0|u_0|^2 \dd x.
\end{align}

In this paper, we work in two spatial dimensions, \ie, $d=2,$ and we make the following assumptions on the initial condition
\begin{equation}\label{ass:data1}
\begin{aligned} 
    & 0 < c_0 \leq \rho_0(x) \leq C_0 < \infty \quad \text{for a.e.~$x \in \R^2$}, \quad u_0 \in  L^2_\sigma(\R^2).
\end{aligned}
\end{equation}
Here and in the following, for $p \in [1,\infty)$, $L^p_\sigma(\R^2)$ denotes the space of $p$-integrable weakly divergence-free vector fields. This choice of initial data allows us to study the solvability of \cref{eq:ins} at a critical level of regularity and to make sense of the energy equality \eqref{eq:energyeq} since \eqref{ass:data1} ensures that $\sqrt{\rho_0} u_0 \in L^2(\R^2)$. 
Based on these observations, we define weak solutions as follows.

\begin{definition}[Leray--Hopf weak solution]\label{def:lerayhopfsol}
We say that a pair $(\rho,u)$ is a \textit{Leray--Hopf} weak solution of \cref{eq:ins} with initial data $(\rho_0,u_0)$ 
satisfying \cref{ass:data1} if
    \begin{itemize}
        \item[(i)] The solution satisfies
        \begin{align*}
           & \sqrt{\rho} u \in L^\infty((0,\infty);\,L^2(\R^2)),
            \quad \rho \in L^\infty((0,\infty)\times \R^2),
            \\ &  u \in L^2_{\loc}([0,\infty)\times\R^2)  \quad \text{and} \quad \nabla u \in L^2((0,\infty) \times \R^2); 
        \end{align*}
        \item[(ii)] The pair $(\rho,u)$ is a distributional solution of \cref{eq:ins}:
        \smallbreak
        \begin{itemize}
            \item[$\bullet$] For every  $\varphi \in C_c^\infty([0,\infty) \times \R^2; \R)$,
        \begin{align*}
             \int_0^\infty \int_{\R^2} \rho \partial_s \varphi \dd x \dd s
            +  \int_0^\infty \int_{\R^2} \rho u \cdot \nabla \varphi \dd x \dd s 
            =- \int_{\R^2} \rho_0 \varphi(0) \dd x;
        \end{align*}
        \item[$\bullet$] For every $\varphi \in C_{c,\sigma}^\infty([0,\infty) \times \R^2;\R^2)$,
        \begin{align*}
            \int_0^\infty \int_{\R^2} &\rho \partial_s \varphi \cdot u+
            \rho u \otimes u : \nabla \varphi \dd x \dd s  
            - \nu \int_0^\infty \int_{\R^2} \nabla u : \nabla \varphi \dd x \dd s = -\int_{\R^2} \rho_0 u_0 \cdot \varphi(0)\dd x;
        \end{align*}
        \item[$\bullet$]For every  $\varphi \in C_c^\infty([0,\infty) \times\R^2;\R)$,
        \begin{equation*}
            \int_0^\infty\int_{\R^2} u\cdot \nabla \varphi \dd x\dd s =0;
        \end{equation*}
  \end{itemize}
         \item[(iii)] The energy inequality holds: for every $t>0,$
        \begin{align*}
           \frac{1}{2} \int_{\R^2} \rho(t) |u(t)|^2 \dd x + \nu \int_0^t \int_{\R^2} |\nabla u(s)|^2 \dd x \dd s \leq 
           \frac{1}{2} \int_{\R^2} \rho_0|u_0|^2 \dd x.
        \end{align*}
    \end{itemize}
\end{definition}

\begin{remark}\label{rmk:novacuum}
    In \cite{CrinBaratSkondricViolini2024}, it was shown that the bounds in \eqref{ass:data1} are preserved over time, \ie, for every $t \in (0,\infty),$ we have that 
    \begin{equation*}
        0 < c_0 \leq \rho(t,x) \leq C_0 < \infty \quad \text{for a.e.~$x \in \R^2$}.
    \end{equation*}
\end{remark}

In the case of a constant density, \ie, when $\rho \equiv 1,$ the existence, uniqueness and regularity has been known for a long time, see \cite{Leray1934,Ladyzhenskaya1959,LionsProdi1959}. A prototypical result summarizing this then reads as.

\begin{theorem}[Existence and uniqueness for the two-dimensional homogeneous Navier-Stokes system]\label{thm:globalwellposedness2d}
    For every $u_0 \in L^2_\sigma(\R^2)$, the two-dimensional homogeneous Navier-Stokes equations admit unique global-in-time Leray--Hopf solution $u$ which satisfies the energy equality and we have
\begin{align}\label{eq:smoothness2d}
        u \in C([0,\infty);L^2_\sigma(\R^2)) \cap C^\infty((0,\infty) \times \R^2).
    \end{align}
\end{theorem}

In the case of variable densities, the existence results of Leray--Hopf solutions, in the absence of vacuum, date back to \cite{Kazhikhov1974} and, in the presence of vacuum, to \cite{Simon1990}. A wide variety of existence results can be found in \cite{LionsVol11996}. Many possible configurations of vacuum, unbounded densities and even the case of variable viscosity were covered. Nevertheless, the uniqueness of Leray--Hopf solutions, sometimes referred to as Lions' problem, was open for many years which was eventually resolved in \cite{HaoShaoWeiZhang2024} in the absence of vacuum.
\smallbreak
In this paper we give a new proof an analogue of \cref{thm:globalwellposedness2d} for the system \cref{eq:ins} provided that the initial data $(\rho_0,u_0)$ satisfy \cref{ass:data1}. Before we can state our main result, we need the following definition.

\begin{definition}[Immediately strong solutions]\label{def:Immstrongsol}
 We say that a Leray--Hopf solution \( (\rho, u) \) with initial data \( (\rho_0, u_0) \) satisfying \cref{ass:data1} is an \textit{immediately strong solution} to \cref{eq:ins} if the following quantities are finite:
\begin{align}\label{def:IS}
\begin{aligned}
    & A^0_0(\rho,u) \coloneqq \hspace{-0.3cm} && \esssup_{s \in (0, \infty)} \int_{\R^2} \rho |u|^2 \dd x + \int_0^\infty \int_{\R^2} |\nabla u|^2 \dd x \dd s\\
    & A^0_1(\rho, u) \coloneqq \hspace{-0.3cm} && \esssup_{s \in (0, \infty)} s \int_{\R^2} |\nabla u|^2 \dd x + \int_0^\infty s \int_{\R^2} \left(\rho|\partial_s u|^2 + \rho |\dot{u}|^2 + |\nabla^2 u|^2 + \abs{\nabla P}^2 \right) \dd x \dd s,\\
    & A^0_2(\rho,u) \coloneqq \hspace{-0.3cm} && \esssup_{s \in (0, \infty)} s^2 \int_{\R^2}  \left(\rho|\partial_s u|^2 + \rho|\dot{u}|^2 + |\nabla^2 u|^2 + \abs{\nabla P}^2\right) \dd x \\
    & \hspace{-0.2cm} && + \int_0^\infty s^2 \int_{\R^2}  \left( |\partial_s \nabla u|^2 + |\nabla \dot{u}| \right) \dd x \dd s,\\
    &  A^0_3(\rho,u) \coloneqq \hspace{-0.3cm}  && \esssup_{s \in (0, \infty)} s^3 \int_{\R^2}  |\nabla \dot{u}|^2 \dd x + \int_0^\infty s^3\int_{\R^2} \left( |\nabla^2 \dot{u}|^2 + |\nabla \dot{P}|^2 + \rho |\ddot{u}|^2  \dd x \right) \dd s.
    \end{aligned}
\end{align}
\end{definition}

Here and in the following, we use the notation $\dot{w} \coloneqq D_v w \coloneqq \partial_t w + (v \cdot \nabla) w$ for the material derivative of vector field or scalar field $w$ with respect to a vector field $v.$ Since we only consider the case $w=\rho$ and $v=u$ or $u=v=w,$ we only write $\dot{\rho}$ and $\dot{u}$ instead of $D_u \rho$ and $D_u u$, respectively. 

Let $E \coloneqq L^\infty(0,\infty;L^2(\R^2)) \cap L^2(0,\infty; \dot{H}^1(\R^2))$ be equipped with the norm
\begin{align*}
    \norm{u}_E^2 \coloneqq \sup_{t \in (0,\infty)} \norm{u(t)}^2_2 + \nu \int_0^\infty \norm{\nabla u(s)}^2_2 \dd s, \quad u \in E.
\end{align*}
Furthermore, we introduce the following energy space
\begin{align}
    Z \coloneqq \{ (\rho,u) \in L^\infty \times E : A^0_0(\rho,u), A^0_1(\rho,u), A^0_2(\rho,u)\text{ are finite}\},
\end{align}
and the corresponding norm\footnote{Observe that the quantity $A_3$ does not play any role in the definition of $Z$ or $\norm{\cdot}_Z$, so all the techniques we use in the following could be adopted to models where only quantities like $A^0_0,\dots,A^0_2$ are bounded.}
\begin{align}
    \norm{(\rho,u)}_Z \coloneqq \sup_{i=0,\dots,2} A^0_i(\rho,u)^{1/2}.
\end{align}

Immediately strong solutions are quite natural objects and for each pair $(\rho_0, u_0)$, satisfying \cref{ass:data1}, there exists at least one immediately strong solution $(\rho, u),$ see \cite[Theorem 1.1]{Danchin2024}. The time-weighted functionals for solutions of \eqref{eq:ins} where first introduced by Paicu, Zhang and Zhang in \cite{PaicuZhangZhang2013}. They originally used the following functionals:
\begin{align*}
    & A^\eta_1(t, \rho, u) \coloneqq \sigma(t)^{1-\eta} \int_{\R^2} |\nabla u|^2 \dd x + \int_0^t \sigma(s)^{1-\eta} \int_{\R^2} \left(\rho |\dot{u}|^2 + |\nabla^2 u|^2 + \abs{\nabla P}^2 \right) \dd x \dd s,\\
    & A^\eta_2(t, \rho,u) \coloneqq \sigma(t)^{2-\eta} \int_{\R^2}  \left(\rho|\dot{u}|^2 + |\nabla^2 u|^2 + \abs{\nabla P}^2\right) \dd x + \int_0^t \sigma(s)^{2-\eta} \int_{\R^2}  \left(|\nabla \dot{u}|^2 \right) \dd x \dd s,
\end{align*}
where $\eta \in (0,1]$ and $\sigma(t) \coloneqq \min \{t,1\}$, and showed that, for every $u_0 \in H^\eta(\R^2)$,
\begin{align*}
     \esssup_{t \in (0, \infty)} A^\eta_1(t, \rho, u), A^\eta_1(t, \rho, u) \lesssim \norm{u_0}_{H^\eta(\R^2)}.
\end{align*}

In contrast to the case when the density is constant, there is not reason to believe that the velocity of an immediately becomes smooth in time. In fact, this might be wrong if the density is only in $L^\infty(\R^2).$ The reason why it is possible to get higher order estimates as in \cref{def:Immstrongsol} is that $\dot{\rho} = \partial_t \rho + u \cdot \nabla \rho = 0$ which allows to differentiate \cref{eq:ins} with respect to the material derivative. Now we can state the main result.

\begin{theorem}\label{thm:mainthmrefined}
    Let $(\rho_0,u_0)$ satisfy \cref{ass:data1}. There exists a Leray--Hopf solution $(\rho,u)$ with respect to the initial data $(\rho_0,u_0)$ which satisfies the following properties:
    \begin{enumerate}[label=(\roman*)]
        \item \label{it:imstrong} The solution $(\rho,u)$ is immediately strong in sense of \cref{def:Immstrongsol};
        \item \label{it:continzero} We have $\rho \in C_{w^*}([0,\infty);L^\infty(\R^2))$ and $\rho u, \sqrt{\rho} u, u \in C([0,\infty);L^2(\R^2))$;
        \item \label{it:energyequality} The solution $(\rho,u)$ satisfies the energy equality;
        \item  \label{it:stability} Given a Leray--Hopf solution $(\varrho,v)$ with respect to the initial data $(\rho_0,v_0)$, with $v_0 \in L^2_\sigma(\R^2)$, then there exists a $C>0$, with
        \begin{align*}
            C=C(\nu, c_0, C_0,\norm{u_0}_2),
        \end{align*}
        such that 
        \begin{align}\label{eq:stabilityintro}
            \sup_{t \in (0,\infty)} \norm{(u-v)(t)}_2^2 + \nu \int_0^\infty \norm{\nabla (u-v)(t)}_2^2 \dd t \leq C \norm{u_0-v_0}_2^2.
        \end{align}
    \end{enumerate}
    In particular, if $u_0=v_0$, then $\rho=\varrho$ and $u=v$ almost everywhere on $(0,\infty) \times \R^2.$ Thus, the solution $(\rho,u)$ is unique in the class of Leray--Hopf solutions.
\end{theorem}

In order to prove \cref{thm:mainthmrefined}, it suffices to construct, for every $(\rho_0, u_0)$ satisfying \cref{ass:data1}, a solution which satisfies \eqref{it:imstrong}, \eqref{it:continzero}, \eqref{eq:energyeq} and \eqref{it:stability}. Indeed, once \eqref{it:stability} is established, the uniqueness part of \cref{thm:mainthmrefined} follows immediately. 

Furthermore, the uniqueness part is in fact a \emph{weak-strong uniqueness} result since we can take immediately strong solutions as strong solutions and Leray--Hopf solutions as the class of weak solutions. \cref{thm:mainthmrefined} shows that immediately strong solutions are unique among a much larger class of solutions, namely in the class of Leray--Hopf solutions.

In \cite{HaoShaoWeiZhang2024}, the authors, based on new and groundbreaking techniques, showed the uniqueness of weak solutions of \cref{eq:ins} within a certain class which is slightly smaller than the class of Leray--Hopf solutions. More precisely, they showed the uniqueness of immediately strong solutions, see \cref{def:Immstrongsol}. In the end, it turns out that both classes coincide, see \cref{thm:mainthmrefined}.

Some parts of the proof of \cref{thm:mainthmrefined} are strongly inspired by the result in \cite{HaoShaoWeiZhang2024}, in particular, the techniques used to estimate certain nonlinear terms. However, our approach is very different from the approach in \cite{HaoShaoWeiZhang2024}. We shall mainly focus on the construction of weak solutions. We give a completely new proof of the existence of weak solutions based on approximation with more regular solution emanating from more regular data. We show that the approximating sequence converges strongly to a weak solution. This is the standard procedure to show existence in \cref{thm:globalwellposedness2d}, but in the case of variable densities, at least to the best of our knowledge, this is completely new.

\subsection{Literature overview}

Over the last years, huge progress has been made concerning the uniqueness of strong solutions of \cref{eq:ins}. A typical feature of strong solutions, leading to their uniqueness, is an $L^1$-in-time Lipschitz bound on the velocity $u$:
\begin{align*}
    \nabla u \in L^1(0,\infty;L^\infty(\R^2)).
\end{align*}
 One of the first global and local well-posedness results was proved in \cite{LadyzhenskayaSolonnikov1975}. Given a smooth and bounded domain $\Omega \subset \R^d,$ $d=2,3,$ the authors proved it provided that
\begin{align*}
    \rho_0 \in C^1(\Omega), \quad 0<c_0 \leq \rho_0 \leq C_0, \quad u_0 \in W^{2-2/p,p}(\Omega),\quad p>d.
\end{align*}
A similar result was obtained in \cite{Desjardins1997} by relaxing the assumptions on the initial data. Then, many well-posedness results for \cref{eq:ins} were shown in the so-called critical Besov spaces, \ie, for $p>1$, let 
\begin{align*}
    \rho_0 \in \dot{B}^{d/p}_{p,1}(\R^d) \quad \text{and} \quad u_0 \in \dot{B}^{d/p-1}_{p,1}(\R^d).
\end{align*}
For $p=2,$ the local and global existence and uniqueness of strong solutions under certain condition was shown in \cite{Danchin2003,Danchin2004}. This was later extended in \cite{Abidi2007} to the case $1<p<2d.$ In \cite{DanchinMucha2012}, the authors extended this result to initial densities lying in the multiplier space $\mathcal{M}(\dot{B}^{d/p-1}_{p,1}(\R^d))$. In all the aforementioned results, there is either a smallness condition on the velocity or the initial density is supposed to be sufficiently close to a constant. In \cite{AbidiGui2021}, the authors managed to remove all those restrictions in the case that
\begin{align*}
    \rho_0 \in \dot{B}^{\eps}_{2/\eps,1}(\R^2),\quad u_0 \in \dot{B}^{0}_{2,1}(\R^2).
\end{align*}
See \cite{Xu2022} for a similar result in 3D. Later, in \cite{AbidiGuiZhang2024}, the existence of solutions was extended to
\begin{align*}
    \rho_0 \in \dot{B}^{2/\lambda}_{\lambda,\infty}(\R^2) \cap L^\infty(\R^2),\quad u_0 \in \dot{B}^{2/p-1}_{p,1}(\R^2)\cap L^2(\R^2),\quad \frac{1}{2} < \frac{1}{p} + \frac{1}{\lambda} \leq 1.
\end{align*}
The uniqueness in this framework was proved in \cite{CrinBaratDeNittiSkondricViolini2024}.

A lot of progress was also achieved in the case when the density is only bounded without any further regularity assumptions. One of the first results was shown by Danchin and Mucha in \cite{DanchinMucha2013}. This was later extended in \cite{PaicuZhangZhang2013} to the case $d=2$ and in \cite{ChenZhangZhao2016} to the case $d=3.$ There it is shown that for any positive $\rho_0\in L^\infty(\R^d)$ and $u_0 \in H^s(\R^d)$, $s>d/2-1,$ there is a unique strong solution, under the smallness of the velocity for $d=3$. In the case $d=3,$ the author of \cite{Zhang2020} managed to prove the global existence of a strong solution under the critical assumptions $\rho_0 \in L^\infty(\R^3)$ and if $u_0$ is suitably small in $\dot{B}^{1/2}_{2,1}(\R^3)$. In \cite{DanchinWang2023}, the global existence and uniqueness of strong solutions of \cref{eq:ins} is shown under the assumptions 
\begin{align*}
    \norm{\rho_0-1}_\infty \ll 1, \quad u_0 \in \dot{B}^{d/p-1}_{p,1}(\R^d)\cap L^2(\R^d),\quad p \in (1,d),
\end{align*}
where $d=2,3$, and if $d=3$ a smallness condition on the velocity is necessary. Furthermore, the authors also showed the uniqueness of the solution constructed in \cite{Zhang2020}. In \cite{Danchin2024}, a result analogue to \cite{Zhang2020} is proved in the case $d=2$. Let us also mention that the solutions constructed in \cite{Zhang2020} and \cite{Danchin2024} have a unique $C^1$ flow.  

There has also been a lot of progress pertaining to the global well-posedness of \cref{eq:ins} in the presence of vacuum. In \cite{DanchinMucha2019}, the authors showed the local and global well-posedness of strong solutions provided that
\begin{align*}
    0 \leq \rho_0 \leq C_0, \quad u_0 \in H^1(\Omega). 
\end{align*}
Here, $\Omega$ denotes either a bounded $C^2$-domain or the d-dimensional torus and $d=2,3.$ This was extended in \cite{PrangeTan2023} to the case $\Omega = \R^d.$ More precisely, the authors studied the existence and uniqueness of solutions under certain assumptions on the initial vacuum and it was assumed that $u_0 \in H^1(\R^d).$ One of the scenarios which was analyzed in \cite{PrangeTan2023} is the so-called density patch problem. Here it is assumed that the initial density $\rho_0$ is the indicator function of a sufficiently smooth domain $D.$ In \cite{HaoShaoWeiZhang2}, the density patch problem was solved under more general assumptions on $D$ and, in \cite{HaoShaoWeiZhang2024}, the same authors proved the existence of weak solutions in $\R^2$ under the very general assumptions
\begin{align*}
    0 \leq \rho_0 \leq C_0, \quad u_0 \in L^2(\R^2). 
\end{align*}
Uniqueness was shown if $\rho_0$ is bounded away from $0.$ In \cite{HaoShaoWeiZhang2024}, the authors also proved a Fujita-Kato type result in $\R^3,$ more precisely, they showed the global well-posedness of \cref{eq:ins} under the assumption
\begin{align*}
    0 \leq \rho_0 \leq C_0, \quad u_0 \in \dot{H}^{1/2}(\R^3), \quad \norm{u_0}_{\dot{H}^{1/2}} \leq \eps_0. 
\end{align*}
See also \cite{AdogboMuchaSzlenk2025} for a global well-posedness result with unbounded initial density.

\subsection{Weak vs. strong uniqueness}

As we mentioned, for a long time, the uniqueness was only known for strong solutions. The main reason for this is that most of the proofs of uniqueness rely on an additional Lipschitz bound of the velocity. However, one can hardly expect a Lipschitz control of the velocity of Leray--Hopf solution. Indeed, as pointed out by Danchin in \cite{Danchin2024}, this may even fail for the heat equation. To demonstrate this, let $v$ be solution of the heat equation with respect to a function $v_0$, then it can be shown that
\begin{align*}
    \nabla v \in L^1(0,\infty;L^\infty(\R^2)) \quad \Longleftrightarrow \quad v_0 \in \dot{B}^{-1}_{\infty,1}(\R^2).
\end{align*}
However, the space $L^2(\R^2)$ is not embedded into the homogeneous Besov space $\dot{B}^{-1}_{\infty,1}(\R^2).$ 

Let us now discuss how the uniqueness can be achieved in the case of variable densities. To this end, let $(\rho_1,u_1,P_1)$ and $(\rho_2,u_2,P_2)$ denote again smooth two Leray--Hopf solutions of \cref{eq:ins} with respect to an initial datum $(\rho_0,u_0)$ satisfying \cref{eq:ins}. We discuss two different methods.

\subsubsection{Lagrangian coordinates} 

We start with the so-called transition to \textit{Lagrangian coordinates}. The main idea behind this method is a transition to a new coordinate system which is time-dependent and moves according to the flow associated with the velocity. More precisely, let $X_i \colon [0,\infty) \times \R^2 \to \R^2$ be the flow associated to $u_i$, and define, for every $(t,x) \in [0,\infty) \times \R^2$,
\begin{align*}
    \overline{\rho}_i(t,x) = \rho_i(t,X_i(t,x)), \quad \overline{u}_i(t,x) = u_i(t,X_i(t,x)), \quad \overline{P}_i(t,x) = P_i(t,X_i(t,x)).
\end{align*}
The triplet $(\overline{\rho}_i,\overline{u}_i,\overline{P}_i)$ will satisfy a new system. The crucial observation is that $\rho_i$ is constant along characteristic curves $t \mapsto X_i(t,x)$. Therefore, we have, for every $(t,x) \in [0,\infty) \times \R^2$,
\begin{align*}
    \overline{\rho}_i(t,x) = \rho_i(t,X_i(t,x)) = \rho(0,X^{-1}_i(t,X(t,x))) = \rho_i(0,x)=\rho_0(x).
\end{align*}
Thus, in the new formulation, the density is independent of time which is an advantage when studying uniqueness, see \cite{DanchinMucha2012,PaicuZhangZhang2013,DanchinMucha2019,AbidiGui2021,AbidiGuiZhang2024,Danchin2024} for more details. Let us mention that, when trying to prove that $(\rho_1,u_1,P_1)=(\rho_2,u_2,P_2)$, this approach requires a Lipschitz estimate on both velocities $u_1$ and $u_2$ in order to ensure that the systems for $(\rho_1,u_1,P_1)$ and $(\rho_2,u_2,P_2)$ are equivalent to the transformed systems for $(\overline{\rho}_1,\overline{u}_1,\overline{P}_1)$ and $(\overline{\rho}_2,\overline{u}_2,\overline{P}_2)$. Therefore, it will lead only to uniqueness results in classes of strong solutions where a Lipschitz estimate on the velocity is known to hold true.

\subsubsection{Relative energy method and $W^{-1,4}$-stability}

Another method which can be used to show uniqueness is the relative energy method combined with $W^{-1,4}$-stability estimates\footnote{In \cite{DanchinWang2023}, the authors rely on $H^{-1}$ stability estimates and we refer to their paper for details.}. The relative energy method has a wide range of applications in fluid dynamics, see \cite{Wiedemann2018} for various examples. 
\smallbreak
\textbf{The homogeneous case: $\rho\equiv 1$:} Let us start by recalling how the uniqueness of Leray--Hopf solutions of \cref{eq:ins} is shown when the density is constant in order to understand where the main problems lie when trying to establish the uniqueness of Leray--Hopf solutions in the case of variable densities.

% To this end, let $u_0 \in L^2_\sigma(\R^2)$.
For simplicity, assume that $u_0$ is smooth and let 
\begin{align*}
    u_1,u_2 \in L^\infty(0,\infty;L^2(\R^2)) \cap L^2(0,\infty;\dot{H}^1(\R^2))
\end{align*}
denote two smooth Leray--Hopf solutions with pressure $P_1$ and $P_2$ with respect to an initial datum $u_0$. Denoting $\delta u = u_1 - u_2$ and $\delta P = P_1 - P_2$, one obtains
\begin{align}\label{eq:diffmomentumhomogintro}
    \pt \delta u + (u_1 \cdot \nabla) \delta u + \nabla \delta P - \nu \Delta u = - (\delta u \cdot \nabla) u_2. 
\end{align}
Multiplying \cref{eq:diffmomentumhomogintro} with $\delta u$ and integrating over $(0,t) \times \R^2$, we obtain, for every $t>0$,
\begin{align}\label{eq:energydiffhomog}
    \frac{1}{2} \norm{\delta u(t)}^2_2 + \nu \int_0^t \norm{\nabla \delta u(s)}^2_2 \dd s = - \int_0^t \int_{\R^2} \delta u \otimes \delta u : \nabla u_2 \dd x \dd s.
\end{align}
Employing Ladyzhenskaya's famous inequality:  for $f \in H^1(\R^2)$ and $f \in L^4(\R^2)$, 
\begin{align*}
    \norm{f}_4 \leq C \norm{f}_2^{1/2} \norm{\nabla f}_2^{1/2},
\end{align*}
we can estimate, for every $t>0$, the term on the right-hand side as
\begin{align*}
    - \int_0^t \int_{\R^2} \delta u \otimes \delta u : \nabla u_2 \dd x \dd s \leq C \int_0^t \norm{\nabla u_2}_2^2 \norm{\delta u}_2^2  \dd s + \frac{\nu}{2} \int_0^t \norm{\nabla \delta u}_2^2 \dd s.
\end{align*}
Inserting this into \cref{eq:energydiffhomog}, yields, for every $t>0,$ 
\begin{align*}
    \norm{\delta u(t)}^2_2 + \nu \int_0^t \norm{\nabla \delta u(s)}^2_2 \dd s
    %& \leq C \int_0^t \norm{\nabla u_2}_2^2 \norm{\delta u}_2^2 \dd s\\ 
    & \leq C \int_0^t \norm{\nabla u_2}_2^2 \left( \norm{\delta u(s)}_2^2 + \nu \int_0^s \norm{\nabla \delta u(\tau)}^2_2 \dd \tau \right) \dd s.
\end{align*}
Finally, with Grönwall's inequality, we deduce that, for every $t>0$, 
\begin{align*}
    \norm{\delta u(t)}^2_2 + \nu \int_0^t \norm{\nabla \delta u(s)}^2_2 \dd s = 0.
\end{align*}
This implies that $u_1=u_2.$
\smallbreak
\textbf{The inhomogeneous case $\rho\not\equiv 1$}:  Let us turn to the case of variable densities. Again, we denote by $(\rho_1,u_1,P_1)$ and $(\rho_2,u_2,P_2)$ two smooth solutions of \cref{eq:ins} with respect to some smooth initial data $(\rho_0,u_0).$ We obtain
\begin{align}\label{eq:diffmomentuminhomogintro}
    \rho_1 \pt \delta u + \rho_1 (u_1 \cdot \nabla) \delta u + \nabla \delta P - \nu \Delta u = - \rho_1 (\delta u \cdot \nabla) u_2 - \delta \rho \dot{u}_2,
\end{align}
where we used the notations
\begin{align*}
    \delta \rho = \rho_1 - \rho_2, \quad \delta u = u_1 - u_2, \quad \delta P = P_1 - P_2, \quad \dot{u}_2 = \pt u_2 + (u_2 \cdot \nabla) u_2.
\end{align*}
Multiplying \cref{eq:diffmomentuminhomogintro} by $\delta u$ and integrating over $(0,t) \times \R^2$ yields
\begin{align}\label{eq:energydiffinhomog}
    \frac{1}{2}\norm{\sqrt{\rho_1(t)} \delta u(t)}_2^2 + \nu \int_0^t \norm{\nabla \delta u}^2_2 \dd s  &= -\int_0^t \int_{\R^2} \delta \rho \dot{ u_2} \cdot \delta u -\int_0^t \int_{\R^2}  \rho_1 \delta u \otimes \delta u : \nabla u_2.
\end{align}
Compared to \cref{eq:energydiffhomog}, the right-hand side of \cref{eq:energydiffinhomog} contains the additional term  
\begin{align}\label{eq:badterm}
    \int_0^t \int_{\R^2} \delta \rho \dot{ u_2} \cdot \delta u
\end{align}
which is the main obstacle to show uniqueness of Leray--Hopf solutions with the relative energy method. The main difficulty with \cref{eq:badterm} is that, at first glance, it is not quadratic in $\delta u$, which is problematic to apply Grönwall's inequality. To overcome this issue, it is necessary to derive stability estimates for $\delta \rho$ in a well-chosen norm. Recall that $\delta \rho$ satisfies
\begin{align}\label{eq:diffconvofmass}
    \partial_t \delta \rho + \div(\delta \rho u_2) = \div(\rho_1 \delta u) = \nabla \rho_1 \delta u, \quad \delta \rho(0) = 0.  
\end{align}
Let $X$ denote the flow associated to $u_2$. Then, we have, at least formally due to the solution formula for transport equation, see \cite[Proposition 2.3]{AmbrosioCrippa2014}, for every $t \geq 0,$
\begin{align}\label{eq:formulatransintro}
    \delta \rho(t,X(t,x)) = - \int_0^t \dive\left((\rho_1 \delta u)(\tau,X(\tau,x) \right) \dd \tau.
\end{align}
Based on the solution formula \cref{eq:formulatransintro} and the fact that $X$ is a measure-preserving diffeomorphism, for every $s>0$ and every $\phi \in \dot{W}^{1,4/3}(\R^2)$,
\begin{align}\label{eq:W-14stablipschitz}
\begin{aligned}
    \abs{\skal{\delta \rho(s), \phi}} \leq C \int_0^s \norm{(\rho_1 \delta u)(\tau)}_4 \dd \tau \norm{\nabla \phi}_{4/3}.
\end{aligned}
\end{align}
By inserting $\phi = \dot{u}_2(s) \cdot \delta u(s)$ into \cref{eq:W-14stablipschitz} and integrating from $0$ to $t$, we obtain that
\begin{align}\label{eq:productestimateintro}
\begin{aligned}
    - \int_0^t \int_{\R^d} \delta \rho \dot{u}_2 \cdot \delta u 
    \leq & C \int_0^t \int_0^s \norm{(\rho_1 \delta u)(\tau)}_4 \dd \tau \norm{\nabla(\dot{u}_2(s) \cdot \delta u(s))}_{4/3} \dd s\\
    \leq & C \int_0^t \sup_{\tau \in (0,s)} \norm{\sqrt{\rho_1(\tau)}\delta u(\tau)}^{1/2}_2 \int_0^s \norm{\nabla \delta u(\tau)}_2^{1/2} \dd \tau \norm{\nabla(\dot{u}_2(s) \cdot \delta u(s))}_{4/3} \dd s.
\end{aligned}
\end{align}
See for instance the proof of \cite[Proposition 4.2]{CrinBaratSkondricViolini2024} for more details. Using the product rule in Sobolev spaces, Hölder's inequality, the Gagliardo--Nirenberg inequality and that $\rho_1$ is bounded away from $0$, we obtain that, for every $s>0$, formally
\begin{align*}
    \norm{\nabla(\dot{u}_2(s) \cdot \delta u(s))}_{\frac{4}{3}} \leq & \norm{\nabla \dot{u}_2(s) \cdot \delta u(s) + \nabla \delta u(s) \cdot \dot{u}_2(s)}_{\frac{4}{3}} \leq \norm{\nabla \dot{u}_2(s)}_2 \norm{\delta u(s)}_4 + \norm{\nabla \delta u(s)}_2 \norm{\dot{u}_2(s)}_4\\
    \lesssim & \norm{\nabla \dot{u}_2(s)}_2 \|\sqrt{\rho_1(s)} u(s)\|_2^{\frac{1}{2}}\norm{\nabla \delta u(s)}_2^{\frac{1}{2}} + \norm{\nabla \delta u(s)}_2 \|\sqrt{\rho_2(s)}\dot{u}_2(s)\|_2^{\frac{1}{2}} \norm{\nabla \dot{u}_2(s)}_2^{\frac{1}{2}}.
\end{align*}
By inserting the latter into \eqref{eq:productestimateintro}, we deduce that, for every $t \in (0,\infty)$,
\begin{equation}\label{eq:productestimateintroquadratic}
    \begin{split}
        - \int_0^t \int_{\R^d} \delta \rho \dot{u}_2 \cdot \delta u 
        \leq C \int_0^t & \sup_{\tau \in (0,s)} \norm{\sqrt{\rho_1(\tau)}\delta u(\tau)}^{\frac{1}{2}}_2 \bigg (\norm{\nabla \dot{u}_2(s)}_2 \norm{\sqrt{\rho_1(s)} u(s)}_2^{\frac{1}{2}}\norm{\nabla \delta u(s)}_2^{\frac{1}{2}} \\
        & + \norm{\nabla \delta u(s)}_2 \norm{\sqrt{\rho_2(s)}\dot{u}_2(s)}_2^{\frac{1}{2}} \norm{\nabla \dot{u}_2(s)}_2^{\frac{1}{2}} \bigg ) \int_0^s \norm{\nabla \delta u(\tau)}_2^{1/2} \dd s.
    \end{split}
\end{equation}
In \eqref{eq:productestimateintroquadratic}, we recover a quadratic structure in $\delta u$ or in $\sqrt{\rho_1} \delta u$. From here, one can now find an $f \in L^1_\loc([0,\infty))$ such that, for every $t>0$, 
\begin{align*}
    - \int_0^t \int_{\R^d} \delta \rho \dot{u}_2 \cdot \delta u 
    \leq \int_0^t f(s) \norm{\sqrt{\rho_1} \delta u(t)}_2^2 \dd s + 
    \frac{\nu}{4} \int_0^t \norm{\nabla \delta u(t)}_2^2.
\end{align*}
Then the uniqueness follows from a Grönwall argument. Let us investigate the constant $C$ in more details. In \cite{CrinBaratSkondricViolini2024}, it was shown that, for every $t>0$, 
\begin{align*}
    C(t) = \sup_{\tau,s \in (0,t)}\norm{DX(s,X^{-1}(\tau)) DX^{-1}(\tau)}_\infty.
\end{align*}
Thus, one way to ensure that $C$ is bounded for finite times is to ask that, for every $t>0$, 
\begin{align}\label{eq:L1intimelip}
    \int_0^t \norm{\nabla u(s)}_\infty \dd s < \infty.
\end{align}
Therefore, this approach requires also a Lipschitz bound on the velocity of at least one of the two solutions. However, the approach described is not optimal as the paper \cite{HaoShaoWeiZhang2024} suggests and the above methodology must be refined to derive the desired uniqueness. The authors observed that it is possible to replace the $L^1$-in-time Lipschitz bound in \cref{eq:L1intimelip} by the following bound
\begin{align}\label{eq:L2intimelip}
    \int_0^t s \norm{\nabla u(s)}_\infty^2 \dd s < \infty
\end{align}
in order to estimate the term \cref{eq:badterm}. Let us mention that the term \cref{eq:badterm} was estimated by means of a duality argument rather than by means of a $\dot{W}^{-1,4}$-stability estimate such as in \cref{eq:W-14stablipschitz}. It is known that for each pair of initial data $(\rho_0,u_0)$ satisfying \cref{ass:data1} there exists at least one solution $(\rho,u)$ satisfying \cref{eq:L2intimelip}. In the next section, we discuss a strategy how the approach from \cite{HaoShaoWeiZhang2024} can be used to show a bound like in \cref{eq:W-14stablipschitz} by only using \cref{eq:L2intimelip}.

\subsection{Outline of the paper}
The paper is organized as follows. In \cref{sec:section2}, we give a brief overview of our strategy and discuss the main difficulties which have to be overcome to prove \cref{thm:mainthmrefined}. In \cref{sec:section3}, we recall some theory related to the transport equation and in \cref{sec:section4} we prove \cref{thm:mainthmrefined}.

\section{Strategy of proof}\label{sec:section2}

Let us discuss main steps of how to prove \cref{thm:mainthmrefined}.

\subsection{Improved $W^{-1,4}$-stability}

Let $(\rho_1,u_1,P_1)$ be a Leray--Hopf solution and let $(\rho_2,u_2,P_2)$ an immediately strong solution with respect to initial data $(\rho_0,u_0)$ satisfying \cref{ass:data1}. As discussed in the previous section, we cannot expect any kind of Lipschitz bound on $ u_2.$ However, from \cref{def:Immstrongsol}, it can be easily deduced that
\begin{align}\label{eq:replacelip}
    K_0 \coloneqq \int_0^\infty s \norm{\nabla u_2(s)}^2_\infty \dd s \leq C \norm{(\rho,u)}_Z.
\end{align}
This follows from the fact that, for every $t>0$,
\begin{align}\label{eq:replacelipt}
    t \norm{\nabla u_2(t)}^2_\infty \leq t \norm{\nabla u_2(t)}^{1/2}_2 \norm{\dot{u}_2(t)}_2 \norm{\nabla \dot{u}_2(t)}^{1/2}_2.
\end{align}
For a detailed discussion of \cref{eq:replacelipt} and why the right-hand side of \cref{eq:replacelip} is $L^1$-in-time, see the proof of \cref{lem:flowestimatenonlip}.

Thus, $u_2$ enjoys some Lipschitz regularity away from $0$ and, since $\sup_{t \in (0,\infty)} \sqrt{t} \norm{u(t)}_\infty < \infty$ as a consequence of \cref{def:Immstrongsol}, we can define the generalized flow
\begin{align}\label{eq:ODEflowintro}
    \begin{cases}
        \partial_t X(t,s,x) = u(t,X(t,s,x)),\quad t,s>0\\
        X(s,s,x) = x.
    \end{cases}
\end{align}
Then the mapping $X \colon (0,\infty)^2 \times \R^2 \to \R^2$ is well-defined since, for every $\eps > 0$,
\begin{align*}
    \nabla u \in L^1_\loc([\eps,\infty);L^\infty(\R^2)),
\end{align*}
Employing again the solution formula for solutions of the transport equation, see \cite[Prop. 2.3]{AmbrosioCrippa2014} and also \cref{prop:propflowII} \cref{it:solucontieq} below, we get for every $s>\eps>0,$
\begin{align}\label{eq:formulatwoflow}
    \delta \rho(s,x) = \delta \rho(\eps,X(\eps,s,x)) - \int_\eps^s \div(\rho_1 \delta u)(\tau,X(\tau,s,x)).
\end{align}
As we show in \cref{prop:wminus4stab1}, from \cref{eq:formulatwoflow}, we can deduce, for every $\phi \in \dot{W}^{1,4/3}(\R^2)$,
\begin{align}\label{eq:Wminus14estiphiintro}
\begin{aligned}
    \abs{\skal{\delta \rho(s,\cdot),\phi}} \leq \sup_{\tau \in (0,s)} \norm{\sqrt{\rho_1}\delta u}^{\frac{1}{2}}_2 \norm{\nabla \phi}_{\frac{4}{3}} \int_0^s \exp\left( K_0^{1/2} \abs{\ln(s/\tau)}^{1/2}\right) \norm{\nabla \delta u}_2^{\frac{1}{2}}\dd \tau.
\end{aligned}
\end{align}
Large parts of the proof of the estimate \cref{eq:Wminus14estiphiintro} are very similar to the proof of \cref{eq:W-14stablipschitz}. However, taking the limit $\eps \to 0$ is quite delicate and the rigorous justification requires some, in particular, showing that
\begin{align*}
    \delta \rho(\eps,X(\eps,s,x)) \underset{\eps \to 0}{\longrightarrow} 0.
\end{align*}
suitable sense. Then, by inserting $\phi = \dot{u}_2(s) \cdot \delta u(s)$ in \cref{eq:Wminus14estiphiintro}, we can estimate \cref{eq:badterm}, for every $t>0$, as
\begin{align}\label{eq:productestimateintrorefined}
    \begin{aligned}
    - \int_0^t \int_{\R^2} \delta \rho(s) \dot{u}_2(s) \cdot \delta u(s) \dd s \lesssim & \int_0^t \sup_{\tau \in (0,s)} \norm{\delta u}^{1/2}_2 \cdot s \norm{\nabla (\dot{u}_2(s) \cdot \delta u(s))}_{4/3}\\ 
    & \times \frac{1}{s} \int_0^s \exp\left(K_0^{1/2} \abs{\ln(s/\tau)}^{1/2}\right) \norm{\nabla \delta u}_2^{1/2}\dd \tau \dd s.
    \end{aligned}
\end{align}
Compared to \cref{eq:productestimateintro}, \cref{eq:productestimateintrorefined} contains the factor
\begin{align*}
    \exp\left(K_0^{1/2} \abs{\ln(s/\tau)}^{1/2}\right).
\end{align*}
Recall that in \cref{eq:productestimateintro} only a constant depending on the Lipschitz of $u_2$ appeared. It was first observed in \cite{HaoShaoWeiZhang2024} how the Grönwall argument can be concluded with this additional factor and, thus, how the Lipschitz bound on $u_2$ can be replaced by $K_0$, and as we mentioned $K_0$ is bounded for immediately strong solutions. This observation is the cornerstone to establish the uniqueness in the class of Leray--Hopf solutions.

\subsection{Uniqueness via relative energy}

After all these preliminaries, a procedure to prove uniqueness of Leray--Hopf solution could be the following: First, show the identity, for every $t>\eps>0$,
\begin{align}\label{eq:energydiffinhomog3}
    \frac{1}{2}\norm{\sqrt{\rho_1(t)} \delta u(t)}_2^2 + \nu \int_0^t \norm{\nabla \delta u}^2_2 \dd s  \leq \frac{1}{2}\norm{\sqrt{\rho_1(\eps)} \delta u(\eps)}_2^2 - \int_\eps^t \int_{\R^2} \delta \rho \dot{ u_2} \cdot \delta u - \rho_1 \delta u \otimes \delta u : \nabla u_2.
\end{align}
Note that it is way more convenient to derive \cref{eq:energydiffinhomog3} away from $0$ since the term $\dot{u}_2$ is very singular in $0.$ Second, use the inequality \cref{eq:Wminus14estiphiintro} to deduce that there is an $f \in L^1([0,\infty))$ such that, for every $t> \eps >0$,
\begin{align*}
    - \int_\eps^t \int_{\R^2} \delta \rho \dot{ u_2} \cdot \delta u - \rho_1 \delta u \otimes \delta u : \nabla u_2  \leq \int_0^t f(s) \norm{\delta u(s)}_2^2 \dd s \lesssim \int_0^t f(s) \norm{\sqrt{\rho_1(s)} \delta u(s)}_2^2 \dd s.
\end{align*}
Note that, we used that, for every $t \geq 0,$ $\rho_1(t) \geq c_0 >0$. Finally, Grönwall's inequality implies that, for every $t>0$, $\norm{\delta u(t)}_2^2=0$, which implies uniqueness of Leray--Hopf solutions.

However, let us mention that it is far from obvious to show that
\begin{align}\label{eq:epsto0}
    \norm{\delta u(\eps)}_2^2 \underset{\eps \to 0}{\longrightarrow} 0,
\end{align}
and therefore to conclude that
\begin{align*}
    \frac{1}{2}\norm{\delta u(t)}_2^2 + \nu \int_0^t \norm{\nabla \delta u}^2_2 \dd s  \leq \norm{\delta u(\eps)}_2^2 - \int_0^t \int_{\R^2} \delta \rho \dot{ u_2} \cdot \delta u -\int_0^t \int_{\R^2}  \rho_1 \delta u \otimes \delta u : \nabla u_2,
\end{align*}
which is necessary to perform a Grönwall argument and to obtain $\delta u \equiv 0$ and $\delta \rho \equiv 0$.

At this point, one could use \cite[Thm. 2.2]{LionsVol11996} which guarantees that \cref{eq:epsto0}. Nevertheless, in the next section, we would like to discuss another strategy to show uniqueness. We shall rather focus on how the relative energy method can be used to show the existence of solutions than showing the uniqueness of solutions directly. The idea is to show that existence can be obtained by approximation with more regular solutions with respect to the energy norm. The uniqueness and stability, but also the energy equality and the continuity in time, are consequences of the stability estimates needed to show convergence of the approximating sequence. Although the existence and uniqueness of (immediately strong) solutions is known from the literature, \cite{HaoShaoWeiZhang2024}, it is worth mentioning that our arguments allow to give a very direct and global construction of weak solutions. Our proof makes no use of any abstract Aubin-Lions-type arguments to get compactness. Instead our construction is very explicit which makes it very natural.

\subsection{Strategy to prove \cref{thm:mainthmrefined}}

Let us now describe the strategy how to prove \cref{thm:mainthmrefined} based on the improved $W^{-1,4}$-estimates. Given $(\rho_0,u_0)$ satisfying \cref{ass:data1}, define, for every $n \in \N$, $u^n_0 \coloneqq \varphi_n * u_0$. Here and in the following $(\varphi_n)_{n \in \N}$ denotes the standard mollification sequence. More precisely, let $\varphi \colon \R^2 \to [0,\infty)$ be a smooth, radial and compactly supported function such that 
\begin{align*}
    \int_{\R^2} \varphi(x) \dd x = 1
\end{align*}
and define, for every $n \in \N$ and every $x \in \R^2$, 
\begin{align*}
    \varphi_n(x) \coloneqq \frac{1}{n^2} \varphi(x/n).
\end{align*}
For every $n \in \N,$ let $(\rho_n,u_n)$ be the unique solution of \cref{eq:ins} with respect to the regularized data $(\rho_0,u^n_0).$ 

\begin{proposition}\label{prop:smoothcauchy}[Existence of immediately strong solutions refined] 
Let $(\rho_0,u_0^n)$ and $(\rho_n,u_n)$ be as above. Then there exists a constant $C>0$, independent of $n,m \in \N$, such that 
\begin{align}\label{eq:esticauchyprop}
    \norm{u_n-u_m}_E^2 = \sup_{t \in (0,\infty)} \norm{u_n(t) - u_m(t)}^2_2 + \nu \int_0^\infty \norm{\nabla(u_n-u_m)}^2_2 \dd s \leq C^2 \norm{u^0_n-u^0_m}_2^2. 
\end{align}
In particular, $(u_n)_{n \in \N}$ is a Cauchy sequence in the space $E.$ Thus, there exists a $u \in E$ such that $\norm{u_n-u}_E \to 0$ and there is a $\rho \in C_{w^*}([0,\infty);L^\infty(\R^2))$ such that, for every $t>0,$ 
\begin{align*}
    \rho_n(t) \wstarlim \rho(t) \quad \textrm{ in } L^\infty(\R^2).
\end{align*}
Furthermore, $(\rho,u)$ is an immediately strong solution of \eqref{eq:ins} which satisfies the energy equality and
\begin{align*}
    \rho u, \sqrt{\rho} u, u \in C([0,\infty);L^2(\R^2)).
\end{align*}
\end{proposition}

As we will see, the method which we use to prove \cref{prop:smoothcauchy} will also reveal how to deduce uniqueness of Leray--Hopf solutions. The starting point of the proof of \cref{prop:smoothcauchy} is the following proposition which is a consequence of \cite[Theorem 1.1]{PaicuZhangZhang2013}, \cite[Theorem 1.6]{CrinBaratSkondricViolini2024}, \cite[Lemma 2.9]{CrinBaratDeNittiSkondricViolini2024} and \cite[Section 2]{Danchin2024}.

\begin{proposition}\label{prop:PZZ}
    Assume that $(\rho_0,u_0)$ satisfies \cref{ass:data1} and that additionally $u_0 \in H^1(\R^2).$ Then there exists a unique Leray--Hopf solution $(\rho,u) \in E$ which satisfies the energy equality with
    \begin{align*}
        \rho \in C_{w^*}([0,\infty);L^\infty(\R^2)), \quad u \in C([0,\infty);L^2(\R^2)).
    \end{align*}
    Furthermore, we have, for every $t \geq 0$ and almost every $x \in \R^2$,
    \begin{align*}
        c_0 \leq \rho(t,x) \leq C_0,
    \end{align*}    
    and there exist a monotonically increasing function $f = f(\cdot,c_0,C_0,\nu) \colon [0,\infty) \to [0,\infty)$ such that
    \begin{align*}
        \norm{(\rho,u)}_E \leq f(\norm{u_0}_2).
    \end{align*}
\end{proposition}

The existence of such solutions of \cref{eq:ins} follows \cite[Theorem 1.1]{PaicuZhangZhang2013}, the uniqueness among Leray--Hopf solutions is a consequence of \cite[Theorem 1.6]{CrinBaratSkondricViolini2024} and the energy equality was demonstrated in \cite[Lemma 2.9]{CrinBaratDeNittiSkondricViolini2024}. 

Let us discuss some properties of the function $f$ in \cref{prop:PZZ}. Since $(\rho,u)$ satisfies the energy equality and since $\rho$ is bounded away from $0,$ we get, by Ladyzhenskaya's inequality,
\begin{align*}
    \int_0^\infty \norm{u(s)}_4^4 \dd s \leq C \sup_{t \in [0,\infty)} \cdot \int_0^\infty \norm{\nabla u(s)}^2_2 \dd s \leq \frac{C C_0}{\nu c_0} \norm{u_0}^2_2 \eqqcolon M_0 \norm{u_0}^4_2
\end{align*}
The constant $C>0$ is independent of $c_0,C_0,\nu$ and $\norm{u_0}_2.$ As demonstrated in \cite{Danchin2024}, we have
\begin{align*}
    A_1(\rho,u) & \leq \norm{\sqrt{\rho_0} u_0}^2_2 \exp \left( \int_0^\infty \norm{u(s)}_4^4 \dd s \right)\\
    & \leq \norm{\sqrt{\rho_0} u_0}^2_2 \exp \left( C^2_0 M_0 \norm{u_0}^4_2 \right) \leq M_0 \norm{u_0}^2_2 \exp \left( M_1 \norm{u_0}^4_2 \right),
\end{align*}
where $M_1 \coloneqq \max \{ C^2_0,C^2_0 M_0\}.$ Moreover, it was shown in \cite[Section 2]{Danchin2024} that there are constants $M_2,M_3>0$, independent of $\norm{u_0}_2$, depending only on $c_0,C_0,\nu$, such that 
\begin{align*}
    A_2(\rho,u) \leq M_2 \exp \left( M_2 A_1(\rho,u)^2 \right) \quad \text{ and } \quad A_3(\rho,u) \leq M_3 \exp \left( M_3 A_2(\rho,u)^2 \right).
\end{align*}
From here it is easy to find an explicit representation of the desired function $f.$

Let $(\rho_n,u_n)$ denote the solution of \cref{eq:ins} from \cref{prop:PZZ} with respect to the initial data $(\rho_0,u^n_0)$. Our strategy to prove \cref{thm:mainthmrefined} is to show that $(u_n)_{n \in \N}$ is a Cauchy sequence in the space $E.$ This is based on the relative energy method which we use to derive stability estimates.

Writing, for every $n,m \in \N$, $\delta^m_n \rho \coloneqq \rho_m - \rho_n$ and $\delta^m_n u \coloneqq u_m - u_n$, we have
\begin{align}\label{eq:energydiff}
    \begin{aligned}
    \frac{1}{2} \norm{\sqrt{\rho_m(t)} \delta^m_n u(t)}_2^2 + \nu \int_0^t \norm{\nabla \delta^m_n u}^2_2 \dd s  \leq & - \int_0^t \int_{\R^2} \delta^m_{n} \rho \dot{u}_n \cdot \delta^m_n u +\rho_m \delta^m_n u \otimes \delta^m_n u : \nabla u_n\\
    & + \frac{1}{2}\norm{\sqrt{\rho_0} \delta^m_n u(0)}_2^2.
    \end{aligned}
\end{align}
Due to the higher regularity of $(\rho_n,u_n)$, the inequality \cref{eq:energydiff} holds true even if we replace $(\rho_m,u_m)$ by any Leray--Hopf solution. This is a consequence of \cite[Lemma 4.1]{CrinBaratSkondricViolini2024}.

Next, for every $t>0$ and every $m,n \in \N,$ we define
\begin{align}\label{eq:estimatefnmintro}
    f^m_n(t) \coloneqq \norm{\delta^m_n u(t)}_2^2 + \int_0^t \norm{\nabla \delta^m_n u}^2_2.
\end{align}
With some adaptation of the techniques that were developed in \cite{HaoShaoWeiZhang2024}, we show that, for every $n,m \in \N$, there exist a $C=C(n,m)>0$ and a $g_n \in L^1(0,\infty),$ with $g_n \geq 0$, such that, for every $t > 0$,
\begin{align*}
    f^m_n(t) \leq  C(n,m) \norm{\sqrt{\rho_0} \delta^m_n u(0)}_2^2 + \int_0^t g_n(s) f^m_n(s) \dd s.
\end{align*}
By Grönwall's inequality, we have, for every $t>0$,
\begin{align}
    \sup_{t \in (0,\infty)} f^m_n(t) \leq C(n,m) \norm{\sqrt{\rho_0} \delta^m_n u(0)}_2^2 \exp \left(\int_0^\infty g_n(s) \dd s \right).
\end{align}
One of the key steps to conclude that $(u_n)_{n \in \N}$ is a Cauchy sequence is to show that $C(n,m)$ and $g_n$ can be bounded only by quantities appearing in the definition of immediately strong solutions. This was first observed in \cite{HaoShaoWeiZhang2024}. More precisely, we show that there is a continuous, monotonically increasing function $h \colon [0,\infty) \to [0,\infty)$ such that, for every $n,m \in \N$,
\begin{align*}
    C(n,m) , \int_0^\infty g_n(s) \dd s \leq  h \left( \norm{(\rho_n, u_n)}_E \right). 
\end{align*}
Since the function $f=f(\cdot,c_0,C_0,\nu)$ appearing in \cref{prop:PZZ} is monotonically increasing and since, for every $n \in \N$,
\begin{align*}
    \norm{u^n_0}_2 = \norm{u^n_0 * \varphi_n}_2 \leq \norm{u_0}_2,
\end{align*}
we deduce that, for every $n,m \in \N$, 
\begin{align*}
    C(n,m) , \int_0^\infty g_n(s) \dd s \leq  h(f(\norm{u_0}_2)).
\end{align*}
Thus, we get that there is a constant $C>0$ such that, for every $n,m \in \N$,
\begin{align}\label{eq:esticauchyconcl}
    \sup_{t \in (0,\infty)} \norm{u_n(t)-u_m(t)}^2_2 + \int_0^\infty \norm{\nabla (u_n(s)-u_m(s))}^2_2 \dd s \leq C \norm{u^n_0-u^m_0}^2_2.
\end{align}
We conclude that $(u_n)_{n \in \N}$ is a Cauchy sequence in the space $E$ and, thus that there exists a $u \in E$ such that
\begin{align}\label{eq:convunintro}
    \norm{u_n - u}_E \underset{n \to \infty}{\longrightarrow} 0.
\end{align}
This enables us to conclude that there is a $\rho \in C_{w^*}([0,\infty);L^\infty(\R^2))$ such that, for every $t>0$,
\begin{align}\label{eq:convrhonintro}
    \rho_n(t) \underset{n \to \infty}{\wlim} \rho(t) \quad \text{ in } \quad L^\infty(\R^2).
\end{align}
The convergences in \cref{eq:convunintro} and \cref{eq:convrhonintro} suffice to show that $(\rho,u)$ is an immediately strong solution of \cref{eq:ins} which satisfies the energy equality and all other regularity conditions in \cref{thm:mainthmrefined}.

The advantage of the relative energy method, which we use to derive \cref{eq:esticauchyprop}, is that \cref{eq:esticauchyprop} remains true when we replace $(\rho_m,u_m)$ by any Leray--Hopf solution $(\varrho,v)$ with respect to some initial data $(\rho_0,v_0),$ where $v_0 \in L^2_\sigma(\R^2).$ We obtain, for every $n \in \N$,
\begin{align*}
    \sup_{t \in (0,\infty)} \norm{u_n(t)-v(t)}^2_2 + \nu \int_0^\infty \norm{\nabla(u_n-v)}^2_2 \dd s \leq C \norm{u^n_0-v_0}^2_2.
\end{align*}
Letting $n \to \infty,$ we deduce the uniqueness of Leray--Hopf solutions of \cref{eq:ins} in the absence of vacuum.

\section{Some transport theory}\label{sec:section3}

\subsection{Weak and renormalized solution of the continuity equation}

In this section, we discuss some properties of bounded solutions of the continuity equation. In the following, let $I=[0,\infty)$ or $I=[0,T],$ with $0<T<\infty$, and assume that, for a $q \in [1,\infty]$,
\begin{align}\label{ass:velocity}
    u \in L^1_\loc(I;W^{1,q}_\loc(\R^2)),\quad \frac{u}{1+|x|} \in L^1_\loc (I;L^1(\R^2) + L^\infty(\R^2)), \quad \div u = 0.
\end{align}
Given a $\rho_0 \in L^\infty(\R^2),$ then we consider the continuity equation on $I \times \R^2$
\begin{align}\label{eq:conteq}
    \begin{cases}
        \partial_t \rho + \div(\rho u) =0,\\
        \rho(0) = \rho_0.
    \end{cases}
\end{align}
Due to the incompressibility condition $\div u = 0,$ the first equation in \cref{eq:conteq} is (formally) equivalent to
\begin{align*}
    \partial_t \rho + u \cdot \nabla \rho =0,
\end{align*}
which allows us to use theory about the transport equation.

\begin{definition}[Weak and renormalized solution]\label{def:weakrensol}
    We say that $\rho \in L^\infty(I \times \R^2)$ is a weak of \cref{eq:conteq} if for every $\varphi \in C^\infty_c(I \times \R^2)$,
    \begin{align*}
        \int_I \int_{\R^2} \rho \partial_s \varphi \dd x \dd s +  \int_I \int_{\R^2} \rho u \cdot \nabla \varphi \dd x \dd s = - \int_{\R^2} \rho_0 \varphi(0) \dd x.
    \end{align*}
    We say that $\rho$ is a renormalized solution of \cref{eq:conteq} if, for every $\beta \in C^1(\R),$ $\beta(\rho)$ is a weak solution of \cref{eq:conteq} with initial value $\beta(\rho_0)$.
\end{definition}

\begin{proposition}\label{prop:wellposedtrans}
    Assume that $u$ satisfies \cref{ass:velocity} with $q \in [1,\infty)$. Then, for every $\rho_0 \in L^\infty(\R^2)$ there is a unique renormalized solution $\rho \in L^\infty(I \times \R^2)$, and we can modify $\rho$ on negligible set of times so that
    \begin{align*}
        \rho \in C_{w^*}(I;L^\infty(\R^2)).
    \end{align*}
\end{proposition}

\begin{proof}
    For the existence and uniqueness of weak solutions which are additionally renormalized, see \cite[Cor. II.1 and Cor. II.2]{DiPernaLions1989}. For an argument which shows that weak (renormalized) solutions $\rho \in L^\infty(I \times \R^2)$ of \cref{eq:conteq} can be modified so that 
    \begin{align*}
        \rho \in C_{w^*}(I;L^\infty(\R^2)),
    \end{align*}
    see \cite[Remark 2.2.2]{Crippa2008}.
\end{proof}

\begin{proposition}
    Let $(\rho,u)$ be a Leray--Hopf solution with respect to initial data $(\rho_0,u_0)$ satisfying \cref{ass:data1}. Then $\rho$ is a renormalized solution of \cref{eq:conteq}.
\end{proposition}

\begin{proof}
    Let $(\rho,u)$ be a Leray--Hopf solution of \cref{eq:ins}. We have to verify that $u$ satisfies \cref{ass:velocity}. From \cref{def:lerayhopfsol}, it follows
    \begin{align*}
        u \in L^1_\loc(I;W^{1,2}_\loc(\R^2)).
    \end{align*}
    In \cite[Lemma A.1]{CrinBaratSkondricViolini2024}, it was shown that
    \begin{align*}
        \frac{u}{1+|x|} \in L^1_\loc(I;L^1(\R^2)+L^\infty(\R^2)).
    \end{align*}
    Thus, all conditions from \cref{ass:velocity} are satisfied and the statement follows from \cref{prop:wellposedtrans}.
\end{proof}
From now on, we assume that, for every solution $\rho$ of \cref{eq:conteq},
\begin{align*}
    \rho \in C_{w^*}(I;L^\infty(\R^2)).
\end{align*}
We need the following lemma which is an extension of \cref{prop:wellposedtrans}, as we need to replace $\beta \in C^1(\R)$ by $\beta \in C^1((0,\infty))$.

\begin{lemma}\label{lem:weakcontibeta}
    Let $(\rho,u)$ be a Leray--Hopf solution with respect to initial data $(\rho_0,u_0)$ satisfying \cref{ass:data1}. Then, for every $\beta \in C^1((0, \infty))$, every $\phi \in C^\infty_c(\R^2)$ and every $t>0,$ we have
    \begin{align*}
        \int_{\R^2} \beta(\rho(t)) \phi \dd x - \int_{\R^2} \beta(\rho_0) \phi \dd x = \int_0^t \int_{\R^2} \beta(\rho) u \cdot \nabla \phi \dd x \dd s
    \end{align*}
    In particular, for every $\beta \in C^1((0,\infty)),$
    \begin{align*}
        \beta(\rho) \in C_{w^*}([0,\infty);L^\infty(\R^2)).
    \end{align*}
\end{lemma}

For the proof of \cref{lem:weakcontibeta}, see \cref{sec:appendixB}. We finish this section with the following statement.

\begin{lemma}\label{lem:contieqcontiL2}
    Suppose that $u$ satisfies \cref{ass:velocity} and assume that
    \begin{align*}
        \rho_0 \in L^2(\R^2) \cap L^\infty(\R^2).
    \end{align*}
    Let $\rho \in C_{w^*}([0,T];L^\infty(\R^2))$ be the unique solution of \cref{eq:conteq} with respect to the data $\rho_0$ and $u.$ Then we have that 
    \begin{align*}
        \rho \in C([0,T];L^2(\R^2)).
    \end{align*}
\end{lemma}

\begin{proof}
    As in the proof of the previous lemma, one can show that
    \begin{align*}
        \rho, \rho^2 \in C_{w^*}([0,T];L^\infty(\R^2)).
    \end{align*}
    Using the same truncation argument as in the proof of \cite[Cor. II.1]{DiPernaLions1989}, we get that, for every $t \in [0,T]$, 
    \begin{align*}
        \rho \in C_w([0,T];L^2(\R^2)) \quad \text{ and } \quad \norm{\rho(t)}_2 \in C(I).
    \end{align*}
    Now, it follows by elementary arguments that $\rho \in C([0,T];L^2(\R^2)).$
\end{proof}

\subsection{The flow associated to non-Lipschitz vector fields}

In this section, we study the flow associated to a vector field $u$ satisfying the following assumptions
\begin{align}\label{ass:velocityflowtwo}
    u \in L^1_\loc(0,\infty;H^1(\R^2) \cap W^{1,4}(\R^2)), \quad u \in C_b([\eps,\infty) \times \R^2), \quad u \in L^2 ([\eps,\infty); W^{2,4}(\R^2)).
\end{align}

\begin{proposition}\label{prop:propflowII}
    Assume that $u$ satisfies \cref{ass:velocityflowtwo}. Then there exists a unique mapping $X \in C((0,\infty)^2 \times \R^2)$ such that the following properties hold.
    \begin{enumerate}[label=(\roman*)]
        \item \label{it:ODEflowII} For every $s \in (0,\infty)$ and every $x \in \R^2$, $X(\cdot,s,x) \in C^1((0,\infty))$ and $X(\cdot,s,x)$ satisfies the ODE
        \begin{align}\label{eq:ODEflowII}
            \begin{cases}
                \partial_t X(t,s,x) = u(t,X(t,s,x)),\quad t>0\\
                X(s,s,x) = x.
            \end{cases}
        \end{align}
        \item \label{it:flowC1diffeoII} For every $t,s>0$, $X(t,s.\cdot)$ is $C^1$-diffeomorphism and $X^{-1}(t,s,\cdot)=X(s,t,\cdot)$. Moreover, for every $t,s>0$, we have that $J(t,s) = \det (\nabla_x X(t,s,x)) \equiv 1.$
        \item \label{it:invflowtransportII} For every $t \in (0,\infty),$ the mapping $X(t,\cdot,\cdot) \colon (0,\infty) \times \R^2 \to \R^2$ is in the class $C^1((0,\infty) \times \R^2)$ and 
        \begin{align*}
            \partial_s X(t,s,x) + (u(s,x) \cdot \nabla) X(t,s,x) = 0.
        \end{align*}
        \item \label{it:LinftyboundflowII} For every $t>s>0,$ we have
        \begin{align}\label{eq:LinftyboundflowII}
            \norm{DX(t,s,\cdot)}_\infty \leq \exp \left( \int_s^t \norm{\nabla u(z)}_\infty \dd z \right)
        \end{align}
        \item \label{it:solucontieq} Let $f \in L^2_\loc([0,\infty);L^2(\R^2)) \cap L^\infty((0,\infty) \times \R^2)$. Assume that $\rho \in C_{w^*}([0,\infty); L^\infty(\R^2))$ satisfies, in the sense of distributions,
        \begin{align*}
            \partial_t \rho + \div(\rho u) = f,
        \end{align*}
        Then we have, for every $s,\tau>0$ and almost every $x \in \R^2$,
        \begin{align}\label{eq:solucontieq}
            \rho(s,x) = \rho(\tau,X(\tau,s,x)) + \int_\tau^s f(z,X(z,s,x)) \dd z.
        \end{align}
    \end{enumerate}
\end{proposition}
\noindent
For the proof of \cref{prop:propflowII}, see \cref{sec:appendixB}. Most of the proof is quite elementary, but it is given for the sake of completeness.

\begin{remark}
    Let $f \in L^2_\loc([0,\infty);L^2(\R^2))$ and fix $T>0.$ The mapping 
    \begin{align*}
        [\tau,T] \times \R^2 \ni (s,x) \mapsto \int_\tau^s f(z,X(z,s,x)) \dd z
    \end{align*}
    has to be understood as the $L^2$-limit of the functions
    \begin{align*}
        [\tau,T] \times \R^2 \ni (t,x) \mapsto \int_\tau^s f_n(z,X(z,s,x)) \dd z,
    \end{align*}
    where $(f_n)_{n \in \N}$ is a sequence in $C^\infty_c((\tau,T) \times \R^2)$ such that $f_n \to f$ in $L^2([\tau,T];L^2(\R^2)).$
\end{remark}

We finish this section with an important statement about the backward transport equation
\begin{align}\label{eq:backwardtransport}
    \begin{cases}
        \partial_\tau \rho + \div(\rho u) = 0, \quad \tau<s \\
        \rho(s) = \rho_s.
    \end{cases}
\end{align}
Note that all statements previously stated for \cref{eq:conteq} are also true for \cref{eq:backwardtransport}.
\begin{lemma}\label{lem:transportcontiL1}
    Suppose that $u$ satisfies \cref{ass:velocityflow} and let $\rho \in C_{w^*}([0,s];L^\infty(\R^2))$ be the unique solution of \cref{eq:backwardtransport} with terminal value $\rho_s \in C^\infty_c(\R^2).$ Then there exists an $R>0$ such that, for every $\tau \in (0,s]$,
    \begin{align*}
        \supp \{ \rho(\tau) \} \subset B_R(0).
    \end{align*}
    Furthermore, we also have that $\rho \in C([0,s];L^1(\R^2))$.
\end{lemma}

\begin{proof}
    Fix $R_1>0$ such that $\supp \{ \rho_s \} \subset B_{R_1}(0)$. From the assumption $u \in L^1(0,s;W^{1,4}(\R^2))$, we have that $u \in L^1(0,T;L^{\infty}(\R^2))$, and we can estimate 
    \begin{align*}
        \abs{X(s,\tau,x)} &= \abs{x + \int_\tau^s u(z,X(z,\tau,x))\dd s} \leq  \abs{x} + \int_0^s \norm{u(z)}_\infty \leq  \abs{x} + \int_0^T \norm{u(z)}_{W^{1,4}} \eqqcolon \abs{x} + R_2
    \end{align*}
    Let $\rho$ be the unique solution of 
    \begin{align*}
        \partial_t \rho + \div(\rho u) = 0
    \end{align*}
    such that $\rho(s)=\rho_s.$ Then, by \cref{prop:propflowII} \cref{it:solucontieq}, we have that, for every $\tau > 0$ and almost every $x \in \R^2$,
    \begin{align}\label{eq:formulatransportproof}
        \rho(\tau,x) = \rho_s(X(s,\tau,x)).
    \end{align}
    Note that $\rho_s(X(s,\cdot,\cdot)) \in C((0,s] \times \R^2)$ by \cref{prop:propflowII} \cref{it:invflowtransportII}. Thus, from now on, we can assume that \cref{eq:formulatransportproof} is true for every $\tau \in (0,s]$ and every $x \in \R^2.$
    
    Assume for contradiction that the statement is false. Then, for every $n \in \N,$ there exists a pair $(\tau_n,x_n)$, with $\tau_n \in (0,T]$ and $\abs{x_n} \geq n,$ such that
    \begin{align}\label{eq:notzero}
        \rho(\tau_n,x_n) \neq 0.
    \end{align}
    Pick $n \in \N$ such that $n > R_1 + R_2$ and a pair $(\tau_n,x_n) \in (0,s] \times \R^2$, with $\tau_n \in (0,s]$ and $\abs{x_n} \geq n,$ such that \cref{eq:notzero} is true. We define
    \begin{align*}
        y_n \coloneqq X(s,\tau_n,x_n) = x_n + \int_{\tau_n}^s u(z,X(z,\tau_n,x_n)) \dd z.
    \end{align*}
    Due to \cref{eq:formulatransportproof}, we have
    \begin{align*}
        0 \neq \rho(\tau_n,x_n) = \rho(s,X(s,\tau_n,x_n)) = \rho_s(y_n),
    \end{align*}
    which implies that $y_n \in \supp \{ \rho_s \},$ and we have that $\abs{y_n} \leq R_1.$ From the identity 
    \begin{align*}
        y_n = x_n + \int_{\tau_n}^s u(z,X(z,\tau_n,x_n)) \dd z,
    \end{align*}
    we deduce that
    \begin{align*}
        R_1 & \geq \abs{y_n}
        \geq \abs{x_n} - \abs{\int_{\tau_n}^s u(z,X(z,\tau_n,x_n)) \dd z}
        \geq n - \int_{\tau_n}^s \norm{u(z)}_\infty \dd z
        \geq n - R_2,
    \end{align*}
    which is a contradiction to the assumption $n>R_1+R_2$. 

    Since $\rho_s \in L^2(\R^2) \cap L^\infty(\R^2)$, we deduce by  \cref{lem:contieqcontiL2} that $\rho \in C([0,s];L^2(\R^2)).$ Fix $R>0$ such that, for every $t \in (0,s]$,
    \begin{align*}
        \supp \{ \rho(t) \} \subset B_R(0).
    \end{align*}
    Let $(t_n)_{n \in \N} \subset (0,s]$ and $t_0 \in (0,s]$ such that $t_n \to t_0.$ Then we have
    \begin{align}
        \norm{\rho(t_n)-\rho(t_0)}_{L^1(\R^2)} & = \norm{\rho(t_n)-\rho(t_0)}_{L^1(B_R(0))} \leq R^{1/2} \norm{\rho(t_n)-\rho(t_0)}_{L^2(B_R(0))}\\
        & = R^{1/2} \norm{\rho(t_n)-\rho(t_0)}_{L^2(\R^2)} \underset{n \to \infty}{\longrightarrow} 0.
    \end{align}
    We see that $\rho \in C((0,s];L^1(\R^2))$. For $t_0 = 0,$ we have to argue slightly different. Let $(t_n)_{n \in \N} \subset (0,s]$ such that $t_n \to 0.$ On one hand, the same argument as above shows $(\rho(t_n))_{n \in \N}$ is a Cauchy sequence in $L^1(\R^2)$ with limit $\rho_0.$ On the other hand, by \cref{lem:contieqcontiL2}, we have that $\lim_{n \to \infty} \rho(t_n) = \rho(0)$ in $L^2(\R^2).$ At this point, it is easy to conclude that $\rho_0 = \rho(0),$ which completes the proof.
\end{proof}

\subsection{More estimates on the velocity}

In this section, we collect some estimates on the velocity field of an immediately strong solution. Our goal is to show that the velocity $u$ of an immediately strong solution satisfies \cref{ass:velocityflowtwo}. To this end, let $(\rho,u)$ be an immediately strong solution. From \cite[Prop. 3.2]{Danchin2024}, we know that there is a constant, independent of $(\rho,u)$, such that
\begin{align}\label{eq:Linftyregularity}
     \esssup_{t \in (0,\infty)} t \norm{u(t)}^2_\infty \leq C \norm{(\rho,u)}_Z^2.
\end{align}

\begin{lemma}\label{lem:velocityflow}
    Let $(\rho,u)$ be an immediately strong solution. Then $u$ satisfies \cref{ass:velocityflowtwo}.
\end{lemma}

\begin{proof}
    First we show that, for every $\eps > 0$,
    \begin{align*}
        u \in C_b ([\eps,\infty) \times \R^2).
    \end{align*}
    The boundedness of $u$ follows from \cref{eq:Linftyregularity}. Now we show the $u$ can be modified to continuous on $[\eps,\infty).$ First, since $u \in H^1(\eps,\infty;L^2(\R^2)),$ we can modify $u$, still denotes by $u$, on a negligible set of times such that $u \in C([\eps,\infty);L^2(\R^2)).$ Note that there is a constant $C_1>0$ such that, for every $t_1,t_2 \geq \eps$,
    \begin{align}\label{eq:holdertimeproof}
        \norm{u(t_1)-u(t_2)}_2 \leq C_1 \abs{t_1-t_2}^{1/2}.
    \end{align}
    Since $\esssup_{t \geq \eps} \norm{u(t)}_{H^2} \leq C$, we can find a measurable set $I$ such that $\mathscr{L}^1([0,\infty) \setminus I) = 0$ such that, for every $t \in I$,
    \begin{align}
        u(t) \in H^2(\R^2).
    \end{align}
    Let $A$ be a countable dense subset of $I.$ For every $t \in A$, we can modify $u(t)$, still denoted by $u(t)$, such that, for every $x_1,x_2 \in \R^2$,
    \begin{align*}
        \abs{u(t,x_1) - u(t,x_2)} \leq \norm{u(t)}_{H^2} \abs{x_1 - x_2}^{1/2} \leq C_2 \abs{x_1 - x_2}^{1/2}
    \end{align*}
    Note that even after the modification in space, the property \cref{eq:holdertimeproof} remains true. Using Agmon's inequality, we obtain that, for every $t_1,t_2 \in A$ and every $x \in \R^2,$
    \begin{align*}
        \abs{u(t_1,x_1) - u(t_2,x_2)} 
        & \leq \abs{u(t_1,x_1) - u(t_2,x_1)} + \abs{u(t_2,x_1) - u(t_2,x_2)}\\
        & \lesssim \norm{u(t_1)-u(t_2)}^{1/2}_2 \norm{\nabla^2 u(t_1)- \nabla^2 u(t_2)}^{1/2}_2 + \norm{u(t_2)}_{H^2} \abs{x_1 - x_2}^{1/2}\\
        & \lesssim \abs{t_1 - t_2}^{1/4} + \abs{x_1 - x_2}^{1/2}
    \end{align*}
    Hence, $u$ is Hölder continuous on on $A \times \R^2$ and, therefore, $u$ is also uniformly continuous. Thus, we can extend $u|_{A \times \R^2}$ to a continuous $\widetilde{u}$ on $[t_1,t_2] \times \R^2$. It is easy to see that $u = \widetilde{u}$ almost everywhere on $[t_1,t_2] \times \R^2$ and by repeating this procedure we see that we can modify $u$ to a continuous function on $[\eps,\infty) \times \R^2.$

    Since $u \in L^4(0,\infty;L^4(\R^2)),$ we have $u \in L^1_\loc(0,\infty;L^4(\R^2)).$ Furthermore, we have, by Ladyzhenskaya's inequality, for every $s>0$,
    \begin{align*}
        \norm{\nabla u(s)}_4 \leq C \norm{\nabla u(s)}_2^{1/2} \norm{\nabla^2 u(s)}_2^{1/2},
    \end{align*}
    which implies
    \begin{align*}
        \int_0^\infty s^{1/2} \norm{\nabla u(s)}_4^2 \dd s
        & \leq C \int_0^\infty \norm{\nabla u(s)}_2 s^{1/2} \norm{\nabla^2 u(s)}_2 \dd s\\
        & \leq C \left(\int_0^\infty \norm{\nabla u(s)}_2^2 \dd s \right)^{1/2} \left(\int_0^\infty s \norm{\nabla^2 u(s)}_2^2 \dd s \right)^{1/2} \leq \norm{u}_E^2.
    \end{align*}
    Hence, for every $t>0,$
    \begin{align*}
        \int_0^t \norm{\nabla u(s)}_4 \dd s &=
        \int_0^t s^{-1/4} s^{1/4} \norm{\nabla u(s)}_4 \dd s\\
        & \leq \left( \int_0^t s^{-1/2} \dd s \right)^{1/2} \left(\int_0^t s^{1/2} \norm{\nabla u(s)}_4^2 \dd s \right)^{1/2}\\
        & = \sqrt{2} t^{1/4}\left(\int_0^t s^{1/2} \norm{\nabla u(s)}_4^2 \dd s \right)^{1/2},
    \end{align*}
    and we deduce that, for every $t>0$, $u \in L^1(0,t;W^{1,4}(\R^2))$ and that, for every $\eps > 0$, $u \in L^2_\loc([\eps,\infty);W^{1,4}(\R^2)).$ It remains to prove that , for every $\eps > 0$, $\nabla^2 u \in L^2_\loc([\eps,\infty);L^{4}(\R^2)).$ By applying Ladyzhenskaya's inequality again, we obtain
    \begin{align*}
        \int_0^\infty t^4 \norm{\nabla^2 u}^4_4 \dd t \leq \int_0^\infty t^4 \norm{\dot{u}}^4_4 \dd t \leq \int_0^\infty t^2 \norm{\dot{u}}^2_2 t^2 \norm{\nabla \dot{u}}^2_2 \dd t \leq \norm{(\rho,u)}^4_Z
    \end{align*}
    In the first step, we made use of the $L^4$-estimate for the Stokes operator and in the second step we used the definition immediately strong solutions.
\end{proof}

We need the following lemma which can be obtained by the same arguments as \cite[Lemma 5.1]{HaoShaoWeiZhang2024}.

\begin{lemma}\label{lem:flowestimatenonlip}
    Let $(\rho,u)$ be an immediately strong solution of \cref{eq:ins}, and let $X \colon (0,\infty)^2 \times \R^2 \to \R^2$ be the flow associated to $u$. Then, for every $s>\tau>0$,
    \begin{align}
        \norm{DX (s,\tau)}_\infty \leq \exp \left(C \norm{(\rho,u)}_Z \abs{\ln(s/\tau)}^{1/2} \right).
    \end{align}
\end{lemma}

\begin{proof}
    Recall that, for every $v \in W^{1,4}(\R^2)$,
    \begin{align}\label{eq:interpolationinfty}
        \norm{v}_\infty \leq C \norm{v}_4^{1/2} \norm{\nabla v}_4^{1/2},
    \end{align}
    By applying \cref{eq:interpolationinfty} to $\nabla u$, we can estimate, for every $s>\tau>0,$
    \begin{align*}
        \int_\tau^s z \norm{\nabla u}_\infty^2 \dd z \leq & \int_\tau^s z \norm{\nabla u}_4 \norm{\nabla^2 u}_4 \dd z\\
        \leq & \int_\tau^s z \norm{\nabla u}_2^{1/2} \norm{\nabla^2 u}_2^{1/2} \norm{\dot{u}}_4 \dd z\\
        \leq & \int_\tau^s z \norm{\nabla u}_2^{1/2} \norm{\nabla^2 u}_2^{1/2} \norm{\dot{u}}_2^{1/2} \norm{\nabla \dot{u}}_2^{1/2} \dd z\\
        \leq & \int_\tau^s  \norm{\nabla u}_2^{1/2} z^{1/4}\norm{\nabla^2 u}_2^{1/2} z^{1/4} \norm{\dot{u}}_2 z^{1/2} \norm{\nabla \dot{u}}_2 \dd z\\
        \leq & \left( \int_\tau^s \norm{\nabla u}_2^2 \dd z \right)^{1/4}
        \left( \int_\tau^s z \norm{\nabla^2 u}_2^2 \dd z \right)^{1/4}\\
        & \left( \int_\tau^s z \norm{\dot{u}}_2^2 \dd z \right)^{1/4}
        \left( \int_\tau^s z^2 \norm{\nabla \dot{u}}_2^2 \dd z \right)^{1/4}\\
        \leq & \norm{(\rho,u)}^2_Z.
    \end{align*}
    Using \cref{prop:propflowII} \cref{it:LinftyboundflowII}, we have, for every $\tau \in (0,s),$
    \begin{align*}
        \norm{DX(s,\tau)}_\infty &\leq \exp \left(\int_\tau^s z^{-1/2} z^{1/2} \norm{\nabla u}_\infty \dd z \right)\\
        & \leq \exp \left( \left(\int_\tau^s z^{-1} \dd z \right)^{1/2} 
        \left( \int_\tau^s z \norm{\nabla u}_\infty^2 \dd z \right)^{1/2}\right)\\
        & \leq \exp \left( \abs{\ln(s/\tau)}^{1/2} \norm{(\rho,u)}_Z \right)
    \end{align*}
    In the last step, we used that, for every $s > \tau > 0$,
    \begin{align*}
        \int^s_\tau z^{-1} \dd z = \ln(s) - \ln(\tau) = \ln (s/\tau) = \abs{\ln (s/\tau)},
    \end{align*}
    and the previously shown estimate.
\end{proof}

\subsection{$W^{-1,4}$-estimates}

In the following, we assume that $(\rho_1,u_1)$ is a Leray--Hopf and $(\rho_2,u_2)$ is an immediately strong solution with respect to the initial data
\begin{equation}\label{ass:datatwo}
\begin{aligned} 
    0 < c_0 \leq \rho_0(x) \leq C_0 < \infty \quad \text{for a.e.~$x \in \R^2$}, \quad u_0^1,u^2_0 \in  L^2_\sigma(\R^2).
\end{aligned}
\end{equation}
In the following, we write
\begin{align*}
    \delta \rho = \rho_1 - \rho_2, \quad \delta u = u_1 - u_2.
\end{align*}

\begin{proposition}\label{prop:wminus4stab1}
    Let $(\rho_i,u_i)$, $i=1,2$, be as above. Then we have, for every $\phi \in W^{1,4/3}(\R^2)$ and every $s>0,$
    \begin{align}\label{eq:Wminus14estiphi}
    \begin{aligned}
        \abs{\skal{\delta \rho(s,\cdot),\phi}} \leq \sup_{\tau \in (0,s)} \norm{\sqrt{\rho_1}\delta u}^{\frac{1}{2}}_2 \norm{\nabla \phi}_{\frac{4}{3}} \int_0^s \exp\left( \norm{(\rho,u)}_Z \abs{\ln(s/\tau)}^{1/2}\right) \norm{\nabla \delta u}_2^{\frac{1}{2}}\dd \tau.
    \end{aligned}
    \end{align}
\end{proposition}

\begin{proof}
    Let $X \colon (0,\infty)^2 \times \R^2 \to \R^2$ be the flow associated to $u_2$. Furthermore, for every $\gamma > 0$, we write
    \begin{align*}
        \delta \rho_\gamma \coloneqq (\rho_1 - \rho_2) * \varphi_\gamma,
        \quad \textrm{ and } \quad 
        R_\gamma \coloneqq \delta \rho_\gamma u_2 - (\delta \rho u_2)* \varphi_\gamma.
    \end{align*}
    Here, $(\varphi_\gamma)_{\gamma > 0}$ denotes the standard mollification sequence. Since
    \begin{spreadlines}{2ex}
    \begin{equation}\label{eq:transporteqrhogamma}
    \begin{dcases}
    \begin{aligned}
    \partial_s \delta \rho_\gamma + \div(\delta \rho_\gamma u_2) &= - \div(\rho_1 \delta u)_\gamma + \dive(R_\gamma), \qquad  & t>0 \\
    \rho_\gamma|_\eps &= \rho_\gamma(\eps),
    \end{aligned}
    \end{dcases}
    \end{equation}
    \end{spreadlines}
    we can write, due to \cref{prop:propflowII} \cref{it:solucontieq}, for every $s,\eps > 0$,
    \begin{align}\label{eq:formularhogamma}
        \delta \rho_\gamma(s,x) = \delta \rho_\gamma(\eps,X(\eps,s,x)) - \int_\eps^s \div(\rho_1 \delta u)_\gamma(\tau,X(\tau,s,x)) - \dive(R_\gamma)(\tau,X(\tau,s,x))\dd \tau.
    \end{align}
    It suffices to proof \cref{eq:Wminus14estiphi} for every $\phi \in C^\infty_c(\R^2).$ To this end, we multiply \cref{eq:formularhogamma} with $\phi$ and get, for every $s>0$,
    \begin{align*}
        \skal{\delta \rho_\gamma(s),\phi} = \skal{\delta \rho_\gamma(\eps,X(\eps,s)),\phi} - \int_\eps^s \skal{\div(\rho_1 \delta u)_\gamma(\tau,X(\tau,s)) - \dive(R_\gamma)(\tau,X(\tau,s)),\phi} \dd \tau
    \end{align*}
    Then we have, by using a change of coordinates and integration by parts, for every $s>0,$
    \begin{align*}
        \skal{\delta \rho_\gamma(s),\phi(x)} 
        =& \skal{\delta \rho_\gamma(\eps),\phi(X(s,\eps))} - \int_\eps^s \skal{\div(\rho_1 \delta u)_\gamma(\tau,\cdot),\phi(X(s,\tau))} \dd \tau\\
        & + \int_\eps^s \skal{\dive(R_\gamma)(\tau),\phi(X(s,\tau))} \dd \tau\\
        =& \skal{\delta \rho_\gamma(\eps),\phi(X(s,\eps))} + \int_\eps^s \skal{(\rho_1 \delta u)_\gamma(\tau), D X(s,\tau) \nabla \phi(X(s,\tau))} \dd \tau\\
        & + \int_\eps^s \skal{\dive(R_\gamma)(\tau),\phi(X(s,\tau))} \dd \tau\\
    \end{align*}
    We have, for every $\tau \in (\eps,t)$, by \cref{lem:velocityflow},
    \begin{align*}
        \norm{D X(s,\tau)}_\infty \leq \exp \left( \norm{(\rho,u)}_E \abs{\ln(s/\tau)}^{1/2} \right).
    \end{align*}
    Hence, following the arguments in the proof of \cite[Prop. 4.2]{CrinBaratSkondricViolini2024}, we obtain, by passing to the limit $\gamma \to 0$, for every $s>0$,
    \begin{align*}
        \abs{\skal{\delta \rho(s),\phi}} \leq & \abs{\skal{\delta \rho(\eps),\phi(X(s,\eps))}}+ \int_\eps^s \exp \left( \norm{(\rho,u)}_E \abs{\ln(s/\tau)}^{1/2} \right) \norm{\rho_1 \delta u}_4 \norm{\nabla \phi}_{4/3} \dd \tau \\
        \leq & \abs{\skal{\delta \rho(\eps),\phi(X(s,\eps))}}+ \sup_{\tau \in (0,s)} \norm{\sqrt{\rho_1}\delta u}_2^{1/2} s \norm{\nabla \phi}_{4/3}\\ 
        & \frac{1}{s} \int_0^s \exp \left( \norm{(\rho,u)}_E \abs{\ln(s/\tau)}^{1/2} \right) \norm{\nabla \delta u}_2^{1/2} \dd \tau.
    \end{align*}
    Note that we that
    \begin{align*}
        \int_\eps^s \skal{\dive(R_\gamma)(\tau),\phi(X(s,\tau))} \dd \tau \underset{\gamma \to 0}{\longrightarrow} 0.
    \end{align*}
    This follows from the commutator estimates in \cite[Lemma II.1 ii)]{DiPernaLions1989}.
    
    It only remains to argue why 
    \begin{align}\label{eq:remaindereps}
        \skal{\delta \rho(\eps),\phi(X(s,\eps))}\underset{\eps \to 0}{\longrightarrow}0.
    \end{align}
    By \cref{prop:propflowII} \cref{it:solucontieq}, we have that $\widetilde{\phi} \coloneqq \phi(X(s,\cdot,\cdot))$ is the unique renormalized solution of the following backward transport equation
    \begin{spreadlines}{2ex}
    \begin{equation}\label{eq:transportbackward}
    \begin{dcases}
    \begin{aligned}
    \partial_\tau \widetilde{\phi} + u_2 \cdot \nabla \widetilde{\phi} & =0, \qquad  &\tau < s \\
    \widetilde{\phi}(s) &= \phi_s.
    \end{aligned}
    \end{dcases}
    \end{equation}
    \end{spreadlines}
    
    Since $\phi_s \in C^\infty_c(\R^2),$ we have, by \cref{lem:transportcontiL1}, $\phi(X(s,\cdot)) \in C([0,s];L^1(\R^2))$. Thus there is a $\phi_0$ such that $\lim_{\eps \to 0} \phi(X(s,\eps)) = \phi_0$ in $L^1(\R^2)$ as $\eps \to 0$. Using that $\delta \rho(\eps) \wstarlim 0$ in $L^\infty(\R^2)$ as $\eps \to 0,$ we infer that
    \begin{align*}
        \skal{\delta \rho(\eps),\phi_s(X(s,\eps))}\underset{\eps \to 0}{\longrightarrow}0,
    \end{align*}
    which completes the proof.
\end{proof}
    
\section{Proof of \cref{thm:mainthmrefined}}\label{sec:section4}

In this section, we prove \cref{thm:mainthmrefined}. As mentioned above, we use the relative energy method to show stability in $E$. So, given $(\rho_0,u_0)$ as in \cref{ass:data1}, we write , for every $n \in \N$, 
\begin{align*}
    u^n_0 \coloneqq \varphi_n * u_0,
\end{align*}
and let, for every $n \in \N$, $(\rho_n,u_n)$ be the solution of \cref{ass:data1} from \cref{prop:PZZ}. We start with the proof of \cref{prop:smoothcauchy}.

\begin{proof}[Proof of \cref{prop:smoothcauchy}]
    Let $(\rho_0,u_0),(\rho_0,u^n_0)$ and $(\rho_n,u_n)$ be as above. Due to the additional regularity of $(\rho_n,u_n)$, we have by \cite[Lemma 4.1]{CrinBaratSkondricViolini2024}, for every $n \in \N$ and every $t > 0$,
    \begin{align}\label{eq:energydiffnm}
        \begin{aligned}
        \frac{1}{2} & \norm{\sqrt{\rho_m(t)} \delta^m_n u(t)}_2^2 + \nu \int_0^t \norm{\nabla \delta^m_n u}^2_2 \dd s \\ &= -\int_0^t \int_{\R^2} \delta^m_n \rho 
        \dot{u}_n \cdot \delta^m_n u +\rho_m \delta^m_n u \otimes \delta^m_n u : \nabla u_n + \frac{1}{2}\norm{\sqrt{\rho_0} \delta^m_n u(0)}_2^2\\
        & \eqqcolon I^1_{mn}(t) + I^2_{mn}(t)+ \frac{1}{2}\norm{\sqrt{\rho_0} \delta^m_n u(0)}_2^2.
        \end{aligned}
    \end{align}
    Here we used the notation
    \begin{align*}
        \delta^m_n \rho \coloneqq \rho_m-\rho_n, \quad \delta^m_n u \coloneqq u_m-u_n.
    \end{align*}
    Note that \cref{eq:energydiff} remains true for almost every time $t>0$ if we replace $(\rho_m, u_m)$ by any Leray-Hopf solution.
    
    We start by estimating $I^2_{mn}.$ With Hölder's and Ladyzhenskaya's inequality, we infer, for every $t>0$,
    \begin{align*}
        \abs{I^2_{mn}(t)} \leq & C_0 \int_0^t \norm{\nabla u_n}_2 \norm{\delta^m_n u(s)}^2_4 \dd s\\
        \leq & C_0 \int_0^t \norm{\nabla u_n}_2 \norm{\delta^m_n u(s)}_2 \norm{\nabla \delta^m_n u(s)}_2 \dd s\\
        \leq & C_0 C \int_0^t \norm{\nabla u_n}_2^2 \left( \sup_{\tau \leq s} \norm{\delta^m_n u(s)}_2^2 + \int_0^s \norm{\nabla \delta^m_n u}^2_2 \dd \tau \right)\dd s + 
        \frac{\nu}{4}\int_0^t \norm{\nabla \delta^m_n u(s)}^2_2 \dd s.
    \end{align*}
    The constant $C$ in the last line only depends on $\nu.$ Note that in the first step we used that, by \cref{rmk:novacuum}, for every $t>0$ and every $m \in \N$,
    \begin{align}\label{eq:unifbdrhom}
        c_0 \leq \rho_m(t,x) \leq C_0  \quad \text{for~a.e.~$x \in \R^2$}.
    \end{align}
    Next we estimate $I^1_{mn}.$ Using \cref{prop:wminus4stab1} with $\phi = \dot{u}_n(s) \delta^m_n u(s)$, we get, for every $t>0$,
    \begin{align}\label{eq:estimateI1mn1}
    \begin{aligned}
        \abs{I^1_{mn}(t)} \leq & \int_0^t \sup_{\tau \in (0,s)} \norm{\sqrt{\rho_m}\delta^m_n u}^{1/2}_2 \cdot s \norm{\nabla (\dot{u}_n(s) \cdot \delta^m_n u(s))}_{4/3}\\ 
        & \times \frac{1}{s} \int_0^s \exp\left(\norm{(\norm{(\rho_n,u_n)}_Z)} \abs{\ln(s/\tau)}^{1/2}\right) \norm{\nabla \delta^m_n u}_2^{1/2}\dd \tau.
    \end{aligned}
    \end{align}
    Since, for every $n \in \N,$ $\norm{u^0_n}_2 = \norm{\varphi_n * u^0}_2 \leq \norm{u^0}_2,$ we have, by \cref{prop:PZZ}, for every $n \in \N,$
    \begin{align*}
        \norm{(\rho_n,u_n)}_Z \leq f(\norm{u^n_0}_2) \leq f(\norm{u_0}_2).
    \end{align*}
    Thus, we can estimate $I^1_{m,n}$ uniformly, for every $t>0$ and every $m,n \in \N$, as
    \begin{align}\label{eq:estimateI1mn2}
    \begin{aligned}
        \abs{I^1_{mn}(t)} \leq & C_0 \int_0^t \sup_{\tau \in (0,s)} \norm{\delta^m_n u}^{1/2}_2 \cdot s \norm{\nabla (\dot{u}_n(s) \cdot \delta^m_n u(s))}_{4/3}\\ 
        & \times \frac{1}{s} \int_0^s \exp\left(f(\norm{u_0}_2) \abs{\ln(s/\tau)}^{1/2}\right) \norm{\nabla \delta^m_n u}_2^{1/2}\dd \tau.
    \end{aligned}
    \end{align}
    Note that we used the estimate \cref{eq:unifbdrhom} to drop the factor $\sqrt{\rho_m}$.
    
    Using the product rule for Sobolev functions, Hölder's inequality and Ladyzhenskaya's inequality several times, we get, for every $t>0$,  
    \begin{align*}
        \norm{\nabla (\dot{u}_n(s) \cdot \delta^m_n u(s))}_{4/3} 
        \leq &  \norm{\nabla \dot{u}_n(s)}_2\norm{\delta^m_n u(s)}_4 +\norm{\dot{u}_n(s)}_4\norm{\nabla \delta^m_n u(s)}_2\\
        \lesssim & \norm{\nabla \dot{u}_n(s)}_2\norm{\delta^m_n u(s)}_2^{1/2} \norm{\nabla \delta^m_n u(s)}_2^{1/2}\\
        & + \norm{\dot{u}_n(s)}_2^{1/2}  \norm{\nabla \dot{u}_n(s)}_2^{1/2} \norm{\nabla \delta^m_n u(s)}_2\\
        \lesssim  & \norm{\nabla \dot{u}_n(s)}_2 \sup_{\tau \leq s} \norm{\delta^m_n u(\tau)}^{1/2}_2 \norm{\nabla \delta^m_n u(s)}_2^{1/2}\\
        & + \norm{\dot{u}_n(s)}_2^{1/2}  \norm{\nabla \dot{u}_n(s)}_2^{1/2} \norm{\nabla \delta^m_n u(s)}_2.
    \end{align*}
    Collecting the previous estimates in this step, and using Young's inequality with exponents $(4,4,2)$, we conclude, for every $t>0$ and every $m,n \in \N$,
    \begin{align*}
        \abs{I^1_{mn}(t)} \leq & \int_0^t s \norm{\nabla \dot{u}_n(s)}_2 
        \sup_{\tau \leq s} \norm{\delta^m_n u(\tau)}_2  \norm{\nabla \delta^m_n u(s)}_2^{1/2}\\
        & \times \frac{1}{s}\int_0^s \exp\left(f(\norm{u_0}_2)\abs{\ln(s/\tau)}^{1/2}\right)\norm{\nabla \delta^m_n u(\tau)}_2^{1/2} \dd \tau \dd s \\
        & + \int_0^t s \norm{\dot{u}_n(s)}_2^{1/2}  
        \norm{\nabla \dot{u}_n(s)}_2^{1/2}
        \sup_{\tau \leq s} \norm{\delta^m_n u(\tau)}_2^{1/2} 
        \norm{\nabla \delta^m_n u(s)}_2\\
        & \times \frac{1}{s}\int_0^s \exp\left(f(\norm{u_0}_2) \abs{\ln(s/\tau)}^{1/2}\right) \norm{\nabla \delta^m_n u(\tau)}_2^{1/2} \dd \tau \dd s \\
        \leq & C \int_0^t \left(s^2\norm{\nabla \dot{u}_n(s)}_2^2 
        + s^4\norm{\dot{u}_n(s)}_2^2 \norm{\nabla \dot{u}_n(s)}_2^2 \right)
        \sup_{\tau \leq s} \norm{\delta^m_n u(\tau)}_2^2 \dd s\\
        & + \frac{\nu}{8L} \int_0^t \left( \frac{1}{s}\int_0^s \exp\left(f(\norm{u_0}_2) \abs{\ln(s/\tau)}^{1/2}\right)
        \norm{\nabla \delta^m_n u(\tau)}_2^{1/2} \dd \tau \right)^4 \dd s\\ &+ 
        \frac{\nu}{8} \int_0^t \norm{\nabla \delta^m_n u(s)}^2_2 
        \dd s.
    \end{align*}
    The constant $L>0$ will be chosen later, and the constant $C>0$ only depends on $L>0.$ By \cref{prop:PZZ}, we have, for every $n \in \N$ and every $s>0$,
    \begin{align*}
        s^2 \norm{\dot{u}_n(s)}_2^2 \leq \norm{(\rho_n,u_n)}_Z^2 \leq f(\norm{u^n_0}_2) \leq f(\norm{u_0}_2).
    \end{align*}
    According to \cref{lem:Lpinequality}, we can pick the constant $L>0$ such that, for every $t>0$ and every $m,n \in \N$,
    \begin{align*}
        \int_0^t \left( \frac{1}{s}\int_0^s \exp\left(f(\norm{u_0}_2) \abs{\ln(s/\tau)}^{1/2}\right) \norm{\nabla \delta^m_n u(\tau)}_2^{1/2} \dd \tau\right)^4 \dd s \leq L \int_0^t \norm{\nabla \delta^m_n u(s)}^2_2 \dd s.
    \end{align*}
    The constant $L$ only depends on $f(\norm{u_0}_2)$, but neither on $t>0$ nor $m,n \in \N.$ Thus, the constant $C>0$ appearing above is also independent of $t,m,n.$ 
    
    We conclude that there is a $C>0$ such that, for every $t>0$ and every $m,n \in \N,$
    \begin{align*}
        \abs{I^1_{mn}(t)} \leq & C \int_0^t s^2\norm{\nabla \dot{u}_n(s)}_2^2 \sup_{\tau \leq s} \norm{\delta^m_n u(\tau)}_2^2 \dd s + \frac{\nu}{4} \int_0^t \norm{\nabla \delta^m_n u(s)}^2_2 \dd s\\
        \leq & C \int_0^t s^2\norm{\nabla \dot{u}_n(s)}_2^2 \left(\sup_{\tau \leq s} \norm{\delta^m_n u(\tau)}_2^2 + \int^s_0 \norm{\nabla \delta^m_n u} \dd s\right) \dd s+ \frac{\nu}{4} \int_0^t \norm{\nabla \delta^m_n u(s)}^2_2 \dd s
    \end{align*}
    We define, for every $n,m \in \N$ and every $t>0$,
    \begin{align*}
        f^m_n(s) \coloneqq \sup_{\tau \leq s} \frac{1}{2} \norm{\delta^m_n u(\tau)}_2^2 + \int_0^s \norm{\nabla \delta^m_n u}^2_2 \dd \tau,
    \end{align*}
    and by the previous consideration, we have, for every $t>0$ and every $m,n \in \N$,
    \begin{align*}
        f^m_n(t) \leq C\norm{\delta^m_n u(0)}_2^2 + C \int_0^t \left(\norm{\nabla u_n(s)}^2_2 + s^2\norm{\nabla \dot{u}_n(s)}_2^2 \right) f^m_n(s) \dd s.
    \end{align*}
    The constant $C$ only depends on $f(\norm{u_0}_2).$ Since $f^m_n \in L^\infty((0,\infty))$ and 
    \begin{align*}
        g_n(s) \coloneqq \norm{\nabla u_n(s)}^2_2 + s^2\norm{\nabla \dot{u}_n(s)}_2^2 \in L^1((0,\infty)),
    \end{align*}
    more precisely, by \cref{prop:PZZ}, for every $n \in \N,$
    \begin{align*}
        \int_0^\infty g_n(s) \dd s \leq \norm{(\rho_n,u_n)}_Z^2 \leq f(\norm{u_0}_2)^2,
    \end{align*}
    we infer, by employing Grönwall's inequality, for every $t>0$ and every $m,n \in \N$,
    \begin{align*}
        f^m_n(t) \leq C f^m_n(0) \exp\left( C \int_0^t g_n(s) \dd s \right) \leq C \exp\left( C f(\norm{u_0}_2)^2 \right) \norm{u_n(0)-u_m(0)}^2_2.
    \end{align*}
    We conclude that, for every $m,n \in \N$,
    \begin{align}\label{eq:stabcauchyproof}
        \frac{1}{2} \sup_{s \in (0,\infty)} \norm{\delta^m_n u(s)}_2^2 + \nu \int_0^\infty \norm{\nabla \delta^m_n u(s)}^2_2 \leq C \norm{u_n(0)-u_m(0)}^2_2.
    \end{align}
    Note that \cref{eq:stabcauchyproof} remains true if we replace $(\rho_m,u_m)$ by any Leray--Hopf solution $(\rho,u)$ as we made no use of the additional regularity of $(\rho_m,u_m)$. Thus, the sequence $(u_n)_{n \in \N}$ is a Cauchy sequence in the space $E$ and, since $E$ is a Banach space, there is a $u \in E$ such that $u_n \to u$ in $E$ as $n \to \infty.$ 

    Let us now argue why there is a $\rho \in C_{w^*}([0,\infty);L^\infty(\R^2))$ such that 
    \begin{align}\label{eq:convwstarrho}
        \rho_n(t) \underset{n \to \infty}{\wstarlim} \rho(t) \quad \textrm{ in } \quad L^\infty(\R^2). 
    \end{align}
    Since, for every $t>0$, $(\rho_n(t))_{n \in \N}$ is bounded in $L^\infty(\R^2)$ by \cref{eq:unifbdrhom}, by the Banach-Alaoglu Theorem, there is a subsequence $(n_k)_{k \in \N} = (n_k(t))_{k \in \N}$ such that
    \begin{align*}
        \rho_{n_k}(t) \underset{n \to \infty}{\wstarlim} \rho(t) \quad \textrm{ in } \quad L^\infty(\R^2). 
    \end{align*}
    In the next step, we show that, for every $\phi \in C^\infty_c(\R^2)$, $(\skal{\rho_n(t),\phi})_{n \in \N}$ is a Cauchy sequence. From this, due to the density of $C^\infty_c(\R^2)$ in $L^1(\R^2)$, \cref{eq:convwstarrho} follows immediately.
    
    Let $\phi \in C^\infty_c(\R^2)$. By \cref{eq:Wminus14estiphi}, we have, for every $t>0$ and every $n,m \in \N$, 
    \begin{align*}
        \abs{\skal{\rho_n(t)-\rho_m(t),\phi}} \leq & \sup_{s \in (0,t)}
        \norm{u_n(s)-u(s)_m}^{1/2}_2 \norm{\nabla \phi}_{4/3}\\
        & \int_0^t \exp \left( \norm{(\rho_n,u_n)}_Z \abs{\ln(s/t)}^{1/2} \right)\norm{\nabla (u_n-u_m)}_2^{1/2} \dd s\\
        \leq & \sup_{s \in (0,t)}
        \norm{u_n(s)-u(s)_m}^{1/2}_2 \norm{\nabla \phi}_{4/3}\\
        & \left( \int_0^t \exp \left( 2 f(\norm{u_0}_Z) \abs{\ln(s/t)}^{1/2} \right) \right)^{1/2} \left( \int_0^t \norm{\nabla(u_n-u_m)}_2 \right)^{1/2}.
    \end{align*}
    Since $(u_n)_{n \in \N}$ is a Cauchy sequence in $E,$ we deduce, for every $t>0$,
    \begin{align*}
        \abs{\skal{\rho_n(t)-\rho_m(t),\phi}} \underset{n,m \to \infty}{\to} 0.
    \end{align*}
    The same argument shows also that, for every $T>0$,
    \begin{align*}
        \rho_n \wstarlim \rho \quad \text{in} \quad L^\infty((0,T) \times \R^2).
    \end{align*}
    Note that, for every $n \in \N$, every $t>0$ and every $\phi \in C^\infty_c(\R^2)$,
    \begin{align*}
        \int_{\R^2} \rho_n(t) \phi \dd x - \int_{\R^2} \rho_0 \phi \dd x = \int_0^t \int_{\R^2} \rho_n u_n \cdot \nabla \phi \dd x \dd s.
    \end{align*}
    Since, for every $t > 0$,
    \begin{align*}
        \rho_n(t) \underset{n \to \infty}{\wstarlim} \rho(t),\quad \rho_n \underset{n \to \infty}{\wstarlim} \rho \quad \text{in} \quad L^\infty((0,t) \times \R^2),\quad u_n \underset{n \to \infty}{\to} u \quad \text{in} \quad L^2((0,t) \times \R^2),
    \end{align*}
    we can, due to \cref{lem:weaktimesstrong} \cref{it:weaktimesstrong} and \cref{it:weaksqstimesstrong}, pass to the limit $n \to \infty$
    and deduce that, for every $t>0$ and every $\phi \in C^\infty_c(\R^2)$,
    \begin{align}\label{eq:weakformconvmassproof}
        \int_{\R^2} \rho(t) \phi \dd x - \int_{\R^2} \rho_0 \phi \dd x = \int_0^t \int_{\R^2} \rho u \cdot \nabla \phi  \dd x \dd s.
    \end{align}
    From \cref{eq:weakformconvmassproof} it follows immediately that $\rho$ is a weak solution of the conservation of mass equation and that 
    \begin{align*}
        \rho \in C_{w^*}([0,\infty);L^\infty(\R^2)).
    \end{align*}
    Since $(\rho_n,u_n)$ satisfies the energy equality for every $n \in \N$ by \cite[Lemma 2.9]{CrinBaratDeNittiSkondricViolini2024} and since, by \cref{lem:weaktimesstrong} \cref{it:weaktimesstrong} and by what was done before, for every $t>0$, 
    \begin{align*}
        \rho_n(t) u_n(t) \wlim \rho(t) u(t),\quad u_n(t) \to u(t), \quad \nabla u_n \to \nabla u \quad \text{ in } \quad L^2((0,t) \times L^2(\R^2)),
    \end{align*}
    we obtain that, for every $t>0$,
    \begin{align}
        \frac{1}{2}\int_{\R^2} \rho(t)\abs{u(t)}^2 \dd x + \nu \int_0^t \int_{\R^2} \abs{\nabla u(s)}^2 \dd x \dd s =  \frac{1}{2}\int_{\R^2} \rho_0\abs{u_0}^2 \dd x.
    \end{align}
    Hence, $(\rho,u)$ satisfies the the energy equality. It is also clear that, for every $\varphi \in C^\infty_c([0,\infty) \times \R^2)$,
    \begin{align*}
        \int_0^\infty \int_{\R^2} u \cdot \nabla \varphi \dd x \dd s = 0.
    \end{align*}
    From the convergence in $E$, we conclude that
    \begin{align*}
        u_n \to u \quad \text{ in } \quad L^4((0,\infty) \times \R^2),\quad u_n \otimes u_n \to u \otimes u \quad \text{ in } \quad L^2((0,\infty) \times \R^2),
    \end{align*}
    and this implies that we can pass to the limit in the weak formulation of the momentum equation,\ie , we get that, for every $\varphi \in C_{c,\sigma}^\infty([0,\infty) \times \R^2;\R^2)$,
    \begin{align*}
        \int_0^\infty \int_{\R^2} &\rho \partial_s \varphi \cdot u+
        \rho u \otimes u : \nabla \varphi \dd x \dd s  
        - \nu \int_0^\infty \int_{\R^2} \nabla u : \nabla \varphi \dd x \dd s = -\int_{\R^2} \rho_0 u_0 \cdot \varphi(0)\dd x.
    \end{align*}
    Moreover, we have 
    \begin{align*}
        \norm{(\rho,u)}_Z \leq C
    \end{align*}
    as all bounds appearing in the definition of $\norm{\cdot}_Z$ are preserved. Thus, $(\rho,u)$ is also an immediately strong solution.

    Since $\rho \in C_{w^*}([0,\infty);L^\infty(\R^2))$, we have, by \cref{lem:weakcontibeta}, for every $\alpha>0$,
    \begin{align*}
        \rho^\alpha, \rho^{2 \alpha} \in C_{w_*}([0,\infty);L^\infty(\R^2)).
    \end{align*}
    From this for $\alpha=1/2,1$, we conclude, by using \cref{lem:weaktimesstrong} \cref{it:contiproduct}, that
    \begin{align*}
        \rho u, \sqrt{\rho} u \in C([0,\infty);L^2(\R^2)),
    \end{align*}
    which completes the proof.
\end{proof}

Now, we can prove \cref{thm:mainthmrefined}.

\begin{proof}[Proof of \cref{thm:mainthmrefined}]
    It only remains to show the uniqueness and the stability estimate. To this end, let $(\rho_0,u_0)$ satisfy \cref{ass:data1} and let $v_0 \in L^2_\sigma(\R^2)$. Furthermore, let $(\varrho,v)$ be a Leray--Hopf solution of with respect to the initial data $(\rho_0,v_0)$, and let $(\rho_n,u_n)$ and $(\rho,u)$ be as in the previous proof. As we mentioned, the estimate \cref{eq:stabcauchyproof} remains true if we replace $(\rho_m,u_m)$ by $(\varrho,v)$, hence, for every $n \in \N$,
    \begin{align}\label{eq:proofmainN1}
        \frac{1}{2} \sup_{t \in (0,\infty)} \norm{u_n(t)-v(s)}_2^2 + \nu \int_0^\infty \norm{\nabla (u_n-v)}^2_2 \leq C \norm{u^0_n-v_0}^2_2.
    \end{align}
    The previous proof shows that $C$ does not depend on $n.$ By passing to the limit, we deduce that 
    \begin{align}\label{eq:proofmainN2}
        \frac{1}{2} \sup_{t \in (0,\infty)} \norm{u(t)-v(s)}_2^2 + \nu \int_0^\infty \norm{\nabla (u-v)}^2_2 \leq C \norm{u^0-v_0}^2_2.
    \end{align}
    We see, if $u_0=v_0,$ then we have that $u=v$ and, by \cref{eq:Wminus14estiphi} in \cref{prop:wminus4stab1}, we have that, for every $t>0$ and every $\phi \in C^\infty_0(\R^2)$,
    \begin{align}\label{eq:proofmainN3}
        \abs{\skal{\rho(t)-\varrho(t),\phi}} \leq 0,
    \end{align}
    and, therefore, we have that $\rho = \varrho.$

    We conclude that every Leray-Hopf solution with respect to initial data $(\rho_0,u_0)$ satisfying \cref{ass:data1} coincides with the solution constructed in \cref{prop:smoothcauchy}, \ie, Leray--Hopf solutions are unique. Since $(\rho,u)$ is the only Leray--Hopf solution with respect to the initial data $(\rho_0,u)$, the inequality \cref{eq:proofmainN2} is already the desired stability estimate \cref{it:stability}.
\end{proof}

\begin{appendix}

\section{Some technical tools}

Here we collect some auxiliary statements that we need throughout this notes.

\begin{lemma}\label{lem:weaktimesstrong}
    Let $\Omega \subset \R^d$ be measurable. We have that following statements:
    \begin{enumerate}[label=(\roman*)]
        \item \label{it:weaktimesstrong} Assume that $\rho_n \wstarlim \rho$ in $L^\infty(\Omega)$ and that $u_n \to u$ in $L^2(\Omega),$ then
        \begin{align*}
            \rho_n u_n \wlim \rho u \quad \text{in} \quad L^2(\Omega).
        \end{align*}
        \item \label{it:weaksqstimesstrong} Assume that additionally $\rho_n^2 \wlim \rho$ in $L^2(\Omega)$, then
        \begin{align*}
            \rho_n u_n \to \rho u \quad \text{in} \quad L^2(\Omega).
        \end{align*}
        \item \label{it:contiproduct} Let $\rho, \rho^2 \in C_{w^*}([0,\infty);L^\infty(\R^2))$ and $u \in C([0,\infty);L^2(\R^2)),$ then we have
        \begin{align*}
            \rho u \in C([0,\infty);L^2(\R^2)).
        \end{align*}
    \end{enumerate}
\end{lemma}

\begin{proof}
    Statement \cref{it:weaktimesstrong} is elementary. The statement in \cref{it:weaksqstimesstrong} follows from \cref{it:weaktimesstrong} and the observation
    \begin{align*}
        \norm{\rho_n u_n}_2 \to \norm{\rho u}_2.
    \end{align*}
    From the assumptions and \cref{it:weaksqstimesstrong}, it follows directly that $\rho u$ is sequentially continuous with respect to the $L^2(\R^2)$ norm, which completes the proof.
\end{proof}

Let $\Phi \colon \R \to \R$ denote the Gauss error function, \ie,
\begin{align*}
    \Phi(z) \coloneqq \frac{2}{\sqrt{\pi}} \int_0^z \exp(-t^2) \dd t,
\end{align*}
and let $\widetilde{\Phi} = 1 - \Phi.$

\begin{lemma}\label{lem:intfC}
    Let $C>0$ and define $f_C \colon (0,1) \to (0,\infty)$ be given by 
    \begin{align*}
        f_C(x) = \exp \left( C \abs{\ln (x)}^{1/2} \right) = \exp \left( C \sqrt{-\ln (x)} \right).
    \end{align*}
    Then $f_C \in L^1(0,1),$ and we have 
    \begin{align*}
        \int_0^1 f_C(x) \dd x = \frac{2+\sqrt{\pi}Ce^{C^2/4}\widetilde{\Phi}(-C/2)}{2}.
    \end{align*}
\end{lemma}

\begin{proof}
    We set, for every $x \in (0,1)$,
    \begin{align*}
        F_C(x) = x \exp \left( C \sqrt{-\ln (x)} \right) - \frac{\sqrt{\pi}C e^{C^2/2}\Phi((2\sqrt{-\ln(x)}-C)/2)}{2}.
    \end{align*}
    A direct computation shows that, for every $x \in (0,1)$,
    \begin{align*}
        F_C'(x) = f_C(x).
    \end{align*}
    and that
    \begin{align*}
        \lim_{x \to 0} F_C(x) = - \frac{\sqrt{\pi}C e^{C^2/2}}{2}, \quad
        \lim_{x \to 1} F_C(x) = 1 - \frac{\sqrt{\pi}C e^{C^2/2}\Phi(-C/2)}{2}
    \end{align*}
    This completes the proof.
\end{proof}

\begin{lemma}[Minkowski's inequality for integrals] \label{lem:minkow}
    Let $(\Omega_1,\mu_1),(\Omega_2,\mu_2)$ be $\sigma$-finite measure spaces, $p \in (1,\infty)$, and let $F \colon \Omega_1 \times \Omega_2$ be measurable. Then we have
    \begin{align*}
        \left( \int_{\Omega_2} \abs{\int_{\Omega_1} F(x,y) \dd \mu_1(y)}^p \dd \mu_2(y) \right)^{1/p} \leq  \int_{\Omega_1} \left( \int_{\Omega_2} \abs{ F(x,y)}^p \dd \mu_1(y)\right)^{1/p} \dd \mu_2(y) 
    \end{align*}
\end{lemma}

\begin{proof}
    See \cite[Prop. 1.3]{BahouriCheminDanchin2011} .
\end{proof}

The following lemma is an extension of \cite[Lemma 5.2]{HaoShaoWeiZhang2024}. For the sake of completeness, we give a proof.

\begin{lemma}\label{lem:Lpinequality}
    Let $C >0, p \in (1,\infty)$ and $t \in (0,\infty).$ Then there is a $L = L(p,C),$ which is independent of $t$ such that, for every $f \in L^p(0,t)$, we have  
    \begin{align*}
        \int_0^t \abs{ \left( \frac{1}{s}\int_0^s \exp\left(C\abs{\ln(\tau/s)}^{1/2}\right) f(\tau) \dd \tau \right)}^p \dd s \leq L \int_0^t \abs{f(\tau)}^p \dd s.
    \end{align*}
\end{lemma}
\begin{proof}
    Let $f \in L^p(0,t).$ We estimate 
    \begin{align*}
        \int_0^t & \abs{ \left( \frac{1}{s}\int_0^s \exp\left(C\abs{\ln(\tau/s)}^{1/2}\right) f(\tau) \dd \tau \right)}^p \dd s\\
        & \leq \int_0^t \abs{ \left( \int_0^1 \exp\left(C\abs{\ln(\tau)}^{1/2}\right) f(s \tau) \dd \tau \right)}^p \dd s\\
        & \leq \left( \int_0^1 \left( \int_0^t \exp\left(C\abs{\ln(\tau)}^{1/2}\right)^p \abs{f(s \tau)}^p \dd s \right)^{1/p} \dd \tau \right)^p\\
        & \leq \left( \int_0^1 \exp\left(C\abs{\ln(\tau)}^{1/2}\right) \tau^{-1/p} \left( \int_0^t \tau \abs{f(s \tau)}^p \dd s \right)^{1/p} \dd \tau \right)^p\\
        & \leq \left( \int_0^1 \tau^{-1/p} \exp\left(C\abs{\ln(\tau)}^{1/2}\right)  \left( \int_0^{t \tau} \abs{f(s)}^p \dd s \right)^{1/p} \dd \tau \right)^p\\
        & \leq \left( \int_0^1 \tau^{-1/p} \exp\left(C\abs{\ln(\tau)}^{1/2}\right) \dd \tau \right)^p\int_0^t \abs{f(s)}^p \dd s.
    \end{align*}
    In the first step, we used a change of variables, in the second step, we used \cref{lem:minkow} and in the fourth step, we used a change of variables again.

    Fix $q \in (1,p)$ and let $q'=1/q-1.$ Then, we can estimate
    \begin{align*}
        \int_0^1 \tau^{-1/p} \exp\left(C\abs{\ln(\tau)}^{1/2}\right) \dd \tau \leq \left( \int_0^1 \tau^{-q/p} \dd \tau \right)^{1/q} \cdot \left( \int_0^1 \exp \left(Cq' \abs{\ln(\tau)}^{1/2} \right) \dd \tau \right)^{1/q'}.
    \end{align*}
    The first factor on the right-hand side is finite since $q/p<1,$ and the second factor is bounded by \cref{lem:intfC}. Thus, we can set 
    \begin{align*}
        L \coloneqq \left( \int_0^1 \tau^{-q/p} \dd \tau \right)^{p/q} \cdot \left( \int_0^1 \exp \left(Cq' \abs{\ln(\tau)}^{1/2} \right) \dd \tau \right)^{p/q'},
    \end{align*}
    and this completes the proof.
\end{proof}

\section{Proof of \cref{lem:weakcontibeta} and \cref{prop:propflowII}}\label{sec:appendixB}

In this section, we give the proofs of \cref{lem:weakcontibeta} and \cref{prop:propflowII}.

\begin{proof}[Proof of \cref{lem:weakcontibeta}]
    Let $(\varphi_\gamma)_{\gamma > 0}$ denote the standard mollification sequence. Since, by \cref{rmk:novacuum}, for every $t \geq 0,$
    \begin{align*}
        0 < c_0 \leq \rho(t,x) \leq C_0 \quad \text{for~a.e.~$x \in \R^2$},
    \end{align*}
    we deduce that, for every $t \geq 0$, 
    \begin{align*}
        0 < c_0 \leq \rho_\gamma (t,x) \leq C_0 \quad \text{and} \quad  \rho_\gamma (t,x) \to \rho(t,x) \quad \text{for~a.e.~$x \in \R^2$},
    \end{align*}
    This implies that, for every $\beta \in C^1((0,\infty))$ and every $t \geq 0$,
    \begin{align*}
        \beta(\rho_\gamma (t,x)) \to \beta(\rho(t,x)) \quad \text{for~a.e.~$x \in \R^2$}.
    \end{align*}
    Furthermore, it is easy to see that 
    \begin{align*}
        \rho_\gamma \in C_{w^*}([0,T];L^\infty(\R^2)).
    \end{align*}

    Let $\psi \in C^\infty_c(\R^2)$ be a smooth radial cut-off function with $\psi \equiv 1$ on $B_{1/2}(0)$ and $\psi \equiv 0$ outside $B_1(0),$ and set, for every $n \in \N$, $\psi_n \coloneqq \psi(\cdot /n).$ Fix $T>0.$ Note that in $L^2([0,T];W^{1,4}_0(B_n(0)))$
    \begin{align*}
        \partial_t(\rho_\gamma \psi_n) = \psi_n \partial_t \rho_\gamma = - \psi_n \div (u \rho)_\gamma
    \end{align*}
    Since
    \begin{align*}
        H^1([0,T];W^{1,4}_0(B_n(0))) \hookrightarrow C([0,T];W^{1,4}_0(B_n(0))) \hookrightarrow C([0,T] \times B_n(0)),
    \end{align*}
    we can modify $\rho_\gamma \psi_n$ to $\widetilde{\rho}_{\gamma,n}$ such that 
    \begin{align*}
        \widetilde{\rho}_{\gamma,n} \in C([0,T] \times B_n(0)).
    \end{align*}
    It is easy to see that $\rho_{\gamma,n} \in C_{w^*}([0,T];L^\infty(\R^2))$, and we get that, for every $t \in [0,T]$,
    \begin{align*}
        \widetilde{\rho}_{\gamma,n}(t) = \psi_n \rho_\gamma(t) \quad \text{ in } \quad L^\infty(\R^2).
    \end{align*}
    Owing to the regularity in space of $\widetilde{\rho}_{\gamma,n}$ and $\psi_n \rho_\gamma,$ we deduce that, for every $\gamma > 0$ and every $n \in \N$,
    \begin{align*}
        \rho_\gamma(t,x) \psi_n(x) = \widetilde{\rho}_{\gamma,n}(t,x) \quad \text{ for every } \quad (t,x) \in [0,T] \times \R^2.
    \end{align*}
    We conclude that, for every $\gamma > 0$,
    \begin{align*}
        \rho_\gamma \in C([0,T] \times \R^2).
    \end{align*}
    Let $\phi \in C^\infty_c(\R^2)$. Now we can show that, for every $t \in [0,T]$,
    \begin{align*}
        \int_{\R^2} \beta(\rho(t)) \dd x - \int_{\R^2} \beta(\rho(0)) \dd x = \int_0^t \int_{\R^2} \beta(\rho) u \cdot \nabla \phi \dd x \dd s.
    \end{align*}
    Fix $R>0$ such that $\supp \{ \phi \} \subset B_R(0),$ and $\Omega_T \coloneqq [0,T] \times \overline{B_R(0))}$. Since
    \begin{align*}
        \rho_\gamma \in H^1([0,T];W^{1,4}(B_R(0))),
    \end{align*}
    by \cite[Lemma 2.2]{Masuda1984}, there is a sequence $(\rho_{\gamma,m}) \subset C^1([0,T] \times B_R(0))$ such that 
    \begin{align}\label{eq:convL2W14}
        \norm{\rho_{\gamma,m}-\rho_\gamma}_{H^1([0,T];W^{1,4}_0(B_R(0)))} \to 0.
    \end{align}
    In particular, we have that 
    \begin{align}\label{eq:convLinfty}
        \norm{\rho_{\gamma,m}-\rho_\gamma}_{L^\infty(\Omega_T)} \to 0.
    \end{align}
    Thus, we can assume, for every $m \in \N$ and every $(t,x) \in [0,T] \times \overline{B_R(0))}$, that $\rho_{m,\gamma} \geq c_0/2.$ Employing \cref{eq:convL2W14}, \cref{eq:convLinfty} and the chain rule, yields
    \begin{align*}
        \int_{\R^2} \beta(\rho_\gamma(t)) \dd x - \int_{\R^2} \beta(\rho_\gamma(0)) \dd x 
        & = \lim_{m \to \infty} \int_{\R^2} \beta(\rho_{\gamma,m}(t)) \dd x - \int_{\R^2} \beta(\rho_{\gamma,m}(0)) \dd x \\
        & = \lim_{m \to \infty} \int_0^T \int_{\R^2} \partial_t \beta(\rho_{\gamma,m}) \phi \dd x \dd t\\
        & = \lim_{m \to \infty} \int_0^T \int_{\R^2} \partial_t \rho_{\gamma,m} \beta'(\rho_{\gamma,m}) \phi \dd x \dd t\\
        & = \int_0^T \int_{\R^2} \partial_t \rho_{\gamma} \beta'(\rho_{\gamma}) \phi \dd x \dd t\\
    \end{align*}
    In the third step, we used that
    \begin{align*}
        \partial_t \rho_{\gamma,m} \to \pt \rho_{\gamma} \quad \text{ and } \quad \beta'(\rho_{\gamma,m}) \to \beta'(\rho_{\gamma}) \quad \text{ in } L^2(\Omega_T). 
    \end{align*}
    The latter follows from the fact $\rho_{\gamma,m}$ is bounded away from $0$ which allows us to conclude
    \begin{align*}
        \norm{\beta'(\rho_{\gamma,m}) \to \beta'(\rho_{\gamma})}_{L^\infty(\Omega_T)} \to 0.
    \end{align*}
    Since, due to the smoothness of $x \mapsto \rho_\gamma(t,x)$,
    \begin{align*}
        \nabla \beta(\rho_\gamma) = \beta'(\rho_\gamma) \nabla \beta(\rho_\gamma)
    \end{align*}
    and since 
    \begin{align*}
        \partial_t \rho_\gamma + u \cdot \nabla \rho_\gamma = R_\gamma,
    \end{align*}
    where , for every $\gamma > 0$,
    \begin{align*}
        R_\gamma \coloneqq u \cdot \nabla \rho_\gamma - \div ((\rho u)_\gamma) = \div(\rho_\gamma u - (\rho u)_\gamma),
    \end{align*}
    we have that 
    \begin{align*}
        \int_{\R^2} \beta(\rho_\gamma(t)) \dd x - \int_{\R^2} \beta(\rho_\gamma(0)) \dd x = \int_0^t \int_{\R^2} \beta(\rho_\gamma) u \cdot \nabla \phi \dd x \dd s + \int_0^t \int_{\R^2} \beta'(\rho_\gamma) R_\gamma \phi \dd x \dd s.
    \end{align*}
    Using that, for every $t \geq 0$,
    \begin{align*}
        \beta(\rho_\gamma(t)) \underset{\gamma \to 0}{\wstarlim} \beta(\rho(t)) \quad \text{ and } \quad R_\gamma \underset{\gamma \to 0}{\longrightarrow} 0 \quad \text{ in } \quad L^2([0,t];L^2_\loc(\R^2)),
    \end{align*}
    and that, for every $t \geq 0$ and every $\gamma$, 
    \begin{align*}
        |\beta'(\rho_\gamma(t)) \phi| \leq C,
    \end{align*}
    we obtain, due to the compact support of $\phi$, for every $t>0,$
    \begin{align*}
        \int_{\R^2} \beta(\rho(t)) \phi \dd x - \int_{\R^2} \beta(\rho(0)) \phi \dd x = \int_0^t \int_{\R^2} \beta(\rho) u \cdot \nabla \phi \dd x \dd s,
    \end{align*}
    which completes the proof.
\end{proof}

Then, we recall some classical properties of the flow associated to a vector field $u$ satisfying the following regularity assumptions, for every $\eps > 0$, 
\begin{align}\label{ass:velocityflow}
    u \in L^1_\loc(0,\infty;H^1(\R^2) \cap W^{1,4}(\R^2)), \quad u \in C_b([\eps,\infty) \times \R^2) \quad \textrm{ and } \quad \nabla u \in C_b ([\eps,\infty) \times \R^2).
\end{align}

\begin{proposition}[Existence, uniqueness and regularity of the flow associated to $u$]\label{prop:propflowI}
    Suppose that $u$ satisfies \cref{ass:velocityflow}. Then there exists a unique mapping $X \colon (0,\infty)^2 \times \R^2 \to \R^2$ satisfying the following properties.
    \begin{enumerate}[label=(\roman*)]
        \item \label{it:ODEflow} For every $s \in (0,\infty)$ and every $x \in \R^2$, $X(\cdot,s,x) \in C^1((0,\infty))$ and $X(\cdot,s,x)$ satisfies the ODE
        \begin{align}\label{eq:ODEflow}
            \begin{cases}
                \partial_t X(t,s,x) = u(t,X(t,s,x)),\quad t>0\\
                X(s,s,x) = x.
            \end{cases}
        \end{align}
        \item \label{it:semigroupprop} For every $t,s,\tau \in (0,\infty)$ and every $x \in \R^2$,
        \begin{align} \label{eq:semigroupprop}
            X(t,\tau,X(\tau,s,x))=X(t,s,x).
        \end{align}
        In particular, for every $s,t \in (0,\infty)$, $X(s,t,\cdot)$ is invertible and we have
        \begin{align*}
            X^{-1}(t,s,\cdot) = X(s,t,x)
        \end{align*}
        \item \label{it:ODEderivative} For every $s \in (0,\infty)$, $X(\cdot,s,\cdot) \in C^1((0,\infty) \times \R^2)$ and we have that
        \begin{align}\label{eq:ODEderivative}
            \pt DX(t,s,x) = Du(t,x) \cdot DX(t,s,x).
        \end{align}
        \item \label{it:ODEjacobian} Let, for every $t,s>0$ and every $x \in \R^2$, $J(t,s,x) \coloneqq \det DX(t,s,x),$ then
        \begin{align}\label{eq:ODEjacobian}
            \pt J(t,s,x) = \div(u(t,x)) \cdot J(t,s,x).
        \end{align}
        \item \label{it:invflowtransport} For every $t \in (0,\infty),$ the mapping $X(t,\cdot,\cdot) \colon (0,\infty) \times \R^2 \to \R^2$ is in the class $C^1((0,\infty) \times \R^2)$ and 
        \begin{align*}
            \partial_s X(t,s,x) + (u(s,x) \cdot \nabla) X(t,s,x) = 0.
        \end{align*}
    \end{enumerate}
\end{proposition}

\begin{proof}
    Since, for every $\eps > 0$,
    \begin{align}
        u \in C_b([\eps, \infty) \times \R^2;\R^2), \quad 
        \nabla u \in C_b([\eps, \infty) \times \R^2;\R^{2 \times 2}),
    \end{align}
    for every $t>0$ and every $x \in \R^2,$ it follows from the Cauchy-Lipschitz theorem that there are times $t_{-},t_{+} \in (0,\infty)$, with $t_{-}<t<t_{+}$, and that there is a unique $X(\cdot,s,x) \in C^1((t_{-},t_{+}))$ satisfying \cref{eq:ODEflow}. In the next step, we show that $t_{-}=0$. The case $t_{+}=\infty$ is similar. Note that, for every $t \in (t_{-},s)$,
    \begin{align*}
        X(t,s,x) = x + \int_s^t u(z,X(z,s,x)) \dd z,
    \end{align*}
    and we conclude that 
    \begin{align*}
        \abs{X(t,s,x)} \leq \abs{x} + \int^s_t \norm{u(z)}_{\infty} \dd z \leq \abs{x} + \int^s_0 \norm{u(z)}_{\infty} \dd z < \infty.
    \end{align*}
    Thus, there this no blow up for finite times since $u \in L^1_\loc([0,\infty);L^\infty)$. This shows \cref{it:ODEflow}. The property \cref{it:semigroupprop} is the direct consequence of the uniqueness of solutions of \cref{eq:ODEflow}. The statements \cref{it:ODEderivative} and \cref{it:ODEjacobian} were discussed in \cite[Sec. 1.3]{Crippa2008}. For the property \cref{it:invflowtransport}, see \cite[Equality (1.15)]{Crippa2008}.
    
    This shows that there is at least one mapping $X \colon (0,\infty)^2 \times \R^2 \to \R^2$ satisfying \cref{it:ODEflow}-\cref{it:invflowtransport}. Assume that $\widetilde{X} \colon (0,\infty)^2 \times \R^2 \to \R^2$ is mapping which satisfies \cref{it:ODEflow}-\cref{it:invflowtransport}. Then, due to the uniqueness of solutions of \cref{eq:ODEflow}, it follows immediately that, for every $s,t \in (0,\infty)$ and every $x \in \R^2$, $X(s,t,x) = \widetilde{X}(s,t,x)$. Hence, there is at most one mapping satisfying \cref{it:ODEflow}-\cref{it:invflowtransport} which completes the proof.
\end{proof}

\begin{proof}[Proof of \cref{prop:propflowII}]
    For every $n \in \N$ and every $t>0,$ we define
    \begin{align*}
        u_n(t) \coloneqq \varphi_n * u(t).
    \end{align*}
    It is easy to see that, for every $n \in \N$, $u_n$ satisfies \cref{ass:velocityflow}. Thus, by \cref{prop:propflowI}, there is a unique flow $X_n$ satisfying \cref{it:ODEflow}-\cref{it:invflowtransport} from \cref{prop:propflowI}. Furthermore, it follows immediately that $u_n \to u$ uniformly on every compact subset of $(0,\infty) \times \R^2.$ Recall also that, for every $0< \sigma \leq \tau$,
    \begin{align*}
        \norm{u_n}_{L^2([\sigma,\tau];W^{2,4})} \leq \norm{u}_{L^2([\sigma,\tau];W^{2,4})} \quad \text{ and } \quad \norm{u_n-u}_{L^2([\sigma,\tau];W^{2,4})} \underset{n \to \infty}{\longrightarrow} 0.
    \end{align*}
    We show that there is an $X \in C((0,\infty)^2 \times \R^2)$ such that, for every $\tau>\sigma>0$ and every compact $K \subset \R^2$,
    \begin{align*}
        X_n \to X \quad \text{in }C_\loc ([\sigma,\tau]^2 \times K) \quad \text{ and } \quad D X_n \to D X \quad \text{in }C ([\sigma,\tau]^2 \times \R^2).
    \end{align*}
    This suffices to deduce \cref{it:ODEflowII}, \cref{it:flowC1diffeoII} and \cref{it:LinftyboundflowII}. Then we discuss \cref{it:invflowtransportII} and \cref{it:solucontieq}. Since $X$ satisfies \cref{it:ODEflowII}, we deduce that for any other mapping $\widetilde{X}$ which satisfies \cref{it:ODEflowII}-\cref{it:solucontieq} coincides with $X.$ 

    We define 
    \begin{align*}
        H_l \coloneqq\{(s,t) \in (0,\infty)^2 : s \leq t \} \quad \text{ and } \quad H_u \coloneqq\{(s,t) \in (0,\infty)^2 : s \geq t \}
    \end{align*}
    First, we show that there is a $Y \in C((0,\infty)^2 \times \R^2)$ such that, for every $0<\sigma < \tau$,
    \begin{align}\label{eq:convergenceDXn}
        \sup_{(t,s) \in [\sigma,\tau]^2} \norm{DX_n(t,s)-Y(t,s)}_\infty \underset{n \to \infty}{\longrightarrow} 0.
    \end{align}
    From now on let $0 < \sigma < \tau$ be fixed. We start by showing that there is a constant $C = C(\sigma,\tau)>0$ such that, for every $n \in \N$,
    \begin{align}\label{eq:boundednessDXn}
        \sup_{(t,s) \in [\sigma,\tau]^2} \norm{DX_n(t,s)}_\infty \leq C.
    \end{align}
    Let for the moment $\abs{\cdot}$ denote the Frobenius norm on $\R^{2 \times 2}$\footnote{For a matrix $A \in \R^{m \times n}$, the Frobenius norm is defined as $\abs{A} = \left(\sum_{i=1}^m \sum_{j=1}^n a^2_{ij} \right)^{1/2}.$} and let 
    \begin{align*}
        S_l \coloneqq H_l \cap [\sigma,\tau]^2 \quad \text{ and } \quad S_u \coloneqq H_u \cap [\sigma,\tau]^2
    \end{align*}
    We show that \cref{eq:boundednessDXn} is true on $S_l$ and then we discuss the case $S_u.$ Integrating \cref{prop:propflowI} \cref{it:ODEderivative} in time, yields, for every $(t,s) \in S_l$ and every $x \in \R^2,$
    \begin{align*}
        \abs{DX_n(t,s,x)} \lesssim_{\sigma, \tau} 1 + \int_s^t \norm{\nabla u(z)}_\infty \abs{DX_n(z,s,x)} \dd z.
    \end{align*}
    Thus, by Grönwall's inequality and by taking the supremum over $(t,s)$ and $x$, we obtain, for every $n \in \N$,
    \begin{align*}
        \sup_{(t,s) \in S_u} \norm{DX_n(t,s)}_\infty \lesssim_{\sigma, \tau} \exp \left( \int_s^t \norm{\nabla u(z)}_\infty \dd z \right) \leq \exp \left( \int_s^t \norm{u(z)}_{W^{2,4}} \dd z \right).
    \end{align*}
    Since, for every $t,s >0$ and every $n \in \N$, $J_n(t,s) \equiv 1,$ we have that 
    \begin{align*}
        (DX_n(t,s,x))^{-1} = \text{adj}(DX_n(t,s,x))=
        \begin{pmatrix}
            (DX_n(t,s,x))_{22} & -(DX_n(t,s,x))_{12} \\
            -(DX_n(t,s,x))_{21} & (DX_n(t,s,x))_{11} \\
        \end{pmatrix}.
    \end{align*}
    In particular, we deduce that $\abs{DX_n(t,s,x)} = \abs{(DX_n(t,s,x))^{-1}}.$ Owing to the inverse mapping theorem, we have, for every $n \in \N$, every $(s,t) \in S_u$ and every $x,y \in \R^2$ with $X_n(t,s,x)=y,$
    \begin{align*}
        DX_n(s,t,y) = (DX_n(t,s,x))^{-1}.
    \end{align*}
    Using that $(t,s) \in S_l$ if and only if $(s,t) \in S_u$ and that $X_n(t,s)$ is surjective, we get, for every $n \in \N$,
    \begin{align*}
        \sup_{(s,t) \in S_l} \norm{DX_n(s,t)}_\infty = \sup_{(t,s) \in S_u} \norm{DX_n(t,s)}_\infty \lesssim_{\sigma, \tau} \exp \left( \int_s^t \norm{\nabla u(z)}_\infty \dd z \right).
    \end{align*}
    This shows \cref{eq:boundednessDXn}.

    Now, we show \cref{eq:convergenceDXn}. Using \cref{prop:propflowI} \cref{it:ODEderivative} and \cref{eq:boundednessDXn}, we have, for every $n,m \in \N$, every $(t,s) \in S_u$ and every $x \in \R^2$,
    \begin{align*}
        \abs{D X_n(t,s,x)-D X_m(t,s,x)}
        \leq & \int_s^t \norm{\nabla u_n(z)- \nabla u_m(z)}_\infty \abs{D X_n(t,s,x)} \dd z\\
        & + \int_s^t \norm{\nabla u_m(z)}_\infty \abs{D X_n(z,s,x)-D X_m(z,s,x)} \dd z\\
        \lesssim & \int_s^t \norm{u_n(z)-u_m(z)}_{W^{2,4}} \dd z \\
        &+ \int_s^t \norm{\nabla u(z)}_{W^{2,4}} \abs{D X_n(z,s,x)-D X_m(z,s,x)} \dd z.
    \end{align*}
    Again, with Gronwall's inequality, we deduce that, for every $n,m \in \N,$
    \begin{align*}
        \sup_{(t,s) \in S_u} \norm{DX_n(t,s)-DX_m(t,s)}_\infty \leq 
        \norm{u_n-u_m}_{L^1(s,t;W^{2,4})} \exp \left(\int_s^t \norm{u(z)}_{W^{2,4}} \dd z \right)
    \end{align*}
    Using the identity, for every $n,m \in \N$, every $t,s \in (0,\infty)$ and every $x \in \R^2$,
    \begin{align*}
        D X_n(t,s,x)-D X_m(t,s,x) = D X_n(t,s,x)(D X_n(s,t,x)-D X_m(s,t,x))D X_m(t,s,x)
    \end{align*}
    and what was done before, we conclude that, for every $n,m \in \N$,
    \begin{align*}
        \sup_{\sigma \leq t \leq s \leq \tau} \norm{DX_n(t,s)-DX_m(t,s)}_\infty \leq 
        \norm{u_n-u_m}_{L^1(s,t;W^{2,4})} \exp \left(3 \int_s^t \norm{u(z)}_{W^{2,4}} \dd z \right).
    \end{align*}
    Hence, there is a $Y \in C((0,\infty)^2 \times \R^2)$ and we have \cref{eq:convergenceDXn}.

    Next, we show that there is an $X \in (0,\infty)^2 \times \R^2$ such that, for every $0 < \sigma < \tau$ and every compact $K \subset \R^2$, $X_n \to X$ uniformly on $[\sigma,\tau]^2 \times K.$ Fix a compact $K \subset \R^2.$ We obtain with \cref{prop:propflowI} \cref{it:ODEflow} that, for every $n \in \N$, every $(t,s) \in [\sigma,\tau]^2 \cap S_u$ and every $x \in K$,  
    \begin{align*}
        \abs{X_n(t,s,x)} \leq \abs{x} + \int_s^t \abs{u_n(z,X_n(t,s,x))} \dd z \leq \abs{x} + \int_\sigma^\tau \norm{u(z)}_\infty \dd z.
    \end{align*}
    On the other hand, if $(t,s) \in [\sigma,\tau]^2 \cap S_u$, we can argue with \cref{prop:propflowI} \cref{it:invflowtransport} to infer a similar estimate. By applying Grönwall's inequality and taking the supremum, we see that there is a constant $C=C(K)>0$ such that, for every $n \in \N$,
    \begin{align*}
        \sup_{(t,s) \in [\sigma,\tau]^2 \times K} \norm {X_n(t,s,x)} \leq C.
    \end{align*}
    Moreover, with the same arguments as above, we get, for every $n,m \in \N$, 
    \begin{align*}
        \abs{X_n(t,s,x)-X_m(t,s,x)} \leq & \int_s^t \norm{u_n(z)-u_m(z)}_\infty \dd z\\
        & + \int_s^t \norm{\nabla u_m(z)}_\infty \abs{X_n(z,s,x)-X_m(z,s,x)} \dd z\\
        \leq & \int_s^t \norm{u_n(z)-u_m(z)}_{W^{2,4}} \dd z \\
        &+ \int_s^t \norm{u(z)}_{W^{2,4}} \abs{X_n(z,s,x)-X_m(z,s,x)} \dd z.
    \end{align*}
    Thus, by Grönwall's inequality, we obtain, for every $(t,s) \in [\sigma,\tau]$ and every $x \in K$,
    \begin{align*}
        \abs{X_n(t,s,x)-X_m(t,s,x)} \leq \norm{u_n-u_m}_{L^1(s,t;W^{2,4})} \exp \left( \int_s^t \norm{\nabla u(z)}_{W^{2,4}} \dd z \right)
    \end{align*}
    Hence, there is an $X \in C((0,\infty)^2 \times \R^2)$ such that $X_n \to X$ in $C_\loc((0,\infty)^2 \times \R^2).$ 
    
    Since $u_n \to u$ uniformly on every compact subset on $(0,\infty) \times \R^2$, it follows immediately that, for every $s \in (0,\infty)$ and every $x \in \R^2$, $X(\cdot,s,x) \in C^1((0,\infty))$ and that $X(\cdot,s,x)$ satisfies \cref{eq:ODEflowII}. This shows that $X$ satisfies \cref{it:ODEflowII}. Due \cref{prop:propflowI} and due the uniform convergence of $X_n$ to $X$, we deduce that, for every $(t,s,\tau) \in (0,\infty)^3$ and every $x \in \R^2$,
    \begin{align*}
        X(t,s,x) = \lim_{n \to \infty} X_n(t,s,x) = \lim_{n \to \infty} X_n(t,\tau,X_n(\tau,s,x)) = X(t,\tau,X(\tau,s,x)).
    \end{align*}
    We obtain the semi group property and, since, for every $(t,s) \in (0,\infty)^2$, $X(t,s)$ is continuously differentiable, we conclude that $X$ satisfies \cref{it:flowC1diffeoII}. The property \cref{it:LinftyboundflowII} follows from the fact that, for every $n \in \N$ and for every $0<s<t$,
    \begin{align*}
        \norm{DX_n(t,s)}_\infty \lesssim \exp \left(\int_s^t \norm{\nabla u_n(z)}_\infty \dd z\right) \leq \exp \left(\int_s^t \norm{\nabla u(z)}_\infty \dd z\right)
    \end{align*}
    and from the local uniform convergence of $(DX_n)_{n \in \N}$ to $DX.$ In order to show \cref{it:invflowtransportII}, recall that, for every $n \in \N$ and every $t>0$,
    \begin{align*}
        \partial_s X_n (t,s,x) = - (u_n(s,x) \cdot \nabla) X_n (t,s,x).
    \end{align*}
    Observe that the right-hand side, considered as function in $s$ and $x$, converges uniformly on compact sets of $(0,\infty) \times \R^2$ to $(u \cdot \nabla) X (t,\cdot,\cdot)$. This shows that $\partial_s X_n(t,\cdot,\cdot)$ converges locally uniformly to $\partial_s X(t,\cdot,\cdot)$ and we deduce that, for every $t>0$, $X(t,\cdot,\cdot) \in C^1((0,\infty) \times \R^2)$ and that
    \begin{align*}
        \partial_s X (t,s,x) + (u(s,x) \cdot \nabla) X (t,s,x) = 0.
    \end{align*}
    This shows \cref{it:invflowtransportII}. 

    Let us now discuss the proof of \cref{eq:solucontieq}. Let $f \in L^2_\loc((0,\infty)) \times \R^2) \cap L^\infty((0,\infty) \times \R^2)$ and assume that $\rho \in C_{w^*}([0,\infty);L^\infty(\R^2))$ satisfies in the sense of distributions
    \begin{align*}
        \partial_s \rho + \div(\rho u) = f.
    \end{align*}
    Fix $\tau > 0.$ We show that, for every $s \in (0,\tau]$ and almost every $x \in \R^2,$
    \begin{align*}
        \rho(s,x) = \rho(\tau,X(\tau,s,x)) + \int_\tau^s f(z,X)(z,s,x) \dd z.
    \end{align*}
    The case $s \in [\tau,\infty)$ works exactly the same. We define the operator
    \begin{align}
        S \colon C^\infty_c((0,\tau) \times \R^2) \to L^2(0,\tau;L^2(\R^2)),\quad 
        Sf(s,x) = \int_\tau^s f(z,X(z,\tau,x)) \dd z.
    \end{align}
    First we show that $S$ is bounded, then we can extend $S$ on $L^2(0,\tau;L^2(\R^2))$ by setting, for every $f \in L^2(0,\tau;L^2(\R^2))$, 
    \begin{align}\label{eq:extSbylim}
        Sf = \lim_{n \to \infty} Sf_n,
    \end{align}
    where $(f_n)_{n \in \N} \subset C^\infty_c((0,\tau) \times \R^2)$ is such that
    \begin{align*}
        f = \lim_{n \to \infty} f_n.
    \end{align*}
    To this end, let $f \in C^\infty_c((0,\tau) \times \R^2)$. Then we have, for every $s \in (0,\tau)$,
    \begin{align*}
        \norm{Sf(s)}_2 &= \left( \int_{\R^2} \left( \int_\tau^s f(z,X(s,z,x)) \dd z \right)^2 \dd x \right)^{1/2}\\
        &= \left( \int_{\R^2} \left( \int_s^\tau f(z,X(s,z,x)) \dd z \right)^2 \dd x \right)^{1/2}\\
        & \leq \int_s^\tau  \left( \int_{\R^2} (f(z,X(s,z,x)))^2 \dd x\right)^{1/2}\dd z\\
        & \leq \int_s^\tau  \left( \int_{\R^2} (f(z,x))^2 \dd x \right)^{1/2} \dd z\\
        & \leq \tau^{3/4} \norm{f}_2
    \end{align*}  
    In the third step, we used \cref{lem:minkow} and in the fourth step, we used, for every $(z,\tau)$, the change of coordinates $x \mapsto X(z,s).$ The $L^2$-norm is invariant under this change of coordinates since $X(z,s)$ is $C^1$ diffeomorphism with $J(z,s) = \det DX(z,s) \equiv 1$.
    
    We see that, for every $f \in C^\infty_c((0,\tau) \times \R^2)$,
    \begin{align*}
        \norm{Sf}_2 \leq \sup_{s \in [0,\tau]} \norm{Sf(s)}_2 \leq \tau^{3/4} \norm{f}_2
    \end{align*}
    We deduce that $S$ can extended to an operator on $L^2(0,\tau;L^2(\R^2))$ via \cref{eq:extSbylim}. Furthermore, by a similar calculation as above, we see that if additionally $f \in L^\infty((0,\tau) \times \R^2)$, then
    \begin{align*}
        \sup_{s \in (0,\tau)} \norm{S f(s)}_\infty \leq C(\tau) \norm{f}_\infty.
    \end{align*}

    Let now $\rho \in L^\infty(\R^2)$ and $f \in L^2(0,\tau;L^2(\R^2))$. Then, $\rho$ be as in \cref{prop:propflowII} \cref{it:solucontieq}. Then, $\rho$ is a weak solution of 
    \begin{align}\label{eq:backwardtransportproof}
        \begin{cases}
            \partial \rho + \div(\rho u) = f, \quad \tau<s \\
            \rho(s) = \rho_s.
        \end{cases}
    \end{align}
    We show that $\widetilde{\rho}$, defined, for every $(s,x) \in (0,\tau] \times \R^2$, by
    \begin{align*}
        \widetilde{\rho}(s,x) \coloneqq \rho(\tau,X(\tau,s,x)) + \int_\tau^s f(z,X(s,z,x)) \dd z
    \end{align*}
    is a weak solution of \cref{eq:backwardtransportproof}.

    To this end, let, for every $n \in \N$,
    \begin{align*}
        \rho_n(\tau,x) \coloneqq (\varphi_n * \rho) (x).
    \end{align*}
    Note that $\rho_n(\tau,x) \to \rho(\tau,x)$ for almost every $x \in \R^2$ as $n \to \infty$ and that, for every $n \in \N$,
    \begin{align*}
        \norm{\rho_n}_\infty \leq \norm{\rho}_\infty.
    \end{align*}
    It follows immediately that, for every $s \in (0,\tau)$ and almost every $x \in \R^2$,
    \begin{align*}
        \rho_n(X(\tau,s,x),x) \to \rho(X(\tau,s,x),x).
    \end{align*}
    Let $(f_n)_{n \in \N} \subset C^\infty_c((0,\tau) \times \R^2)$ be a sequence such that
    \begin{itemize}
        \item $f_n \to f$ in $L^2(0,\tau;L^2(\R^2))$,
        \item $f_n \to f$ for almost every $(s,x) \in (0,\tau) \times \R^2$,
        \item $\norm{f_n}_\infty \leq \norm{f}_\infty$ for every $n \in \N.$
    \end{itemize}
    Then we have that 
    \begin{align*}
        S f_n \underset{n \to \infty}{\longrightarrow} S f \quad \text{ in } \quad L^2(0,\tau;L^2(\R^2)) \quad \text{ and } \quad \norm{S f}_\infty \leq C \norm{f}_\infty.
    \end{align*}
    Set, for every $n \in \N$, every $s \in (0,\tau]$ and every $x \in \R^2,$
    \begin{align*}
        \widetilde{\rho}_n(s,x) \coloneqq \rho_n(\tau,X(\tau,s,x)) + Sf_n(s,x) =  \rho_n(\tau,X(\tau,s,x)) + \int_\tau^s f_n(z,X(s,z,x)) \dd z
    \end{align*}
    Since $X(\tau,\cdot,\cdot) \in C^1((0,\tau] \times \R^2)$ by \cref{prop:propflowII} \cref{it:invflowtransportII}, we get with the chain rule
    \begin{align*}
        \partial_t \widetilde{\rho}_n + u \cdot \nabla \widetilde{\rho}_n = f_n
    \end{align*}
    Using that $u$ is divergence free, we get that, for every $\phi \in C^\infty((0,\tau] \times \R^2)$,
    \begin{align}\label{eq:weakformphin}
        \int_0^\tau \int_{\R^2} \widetilde{\rho}_n \partial_s \phi + \widetilde{\rho}_n u \cdot \nabla \phi \dd x \dd s = \int_{\R^2} \widetilde{\rho}_n(\tau) \phi(\tau) \dd x + \int_0^\tau \int_{\R^2} f_n \phi \dd x \dd s
    \end{align}
    and that, for every $s \in (0,\tau]$ and every $\psi \in C^\infty_c(\R^2),$
    \begin{align}\label{eq:weakformpsin}
        \int_{\R^2} \widetilde{\rho}_n(\tau) \psi \dd x - \int_{\R^2} \widetilde{\rho}_n(s) \psi \dd x = \int_0^\tau \int_{\R^2} \widetilde{\rho}_n u \cdot \nabla \psi \dd x \dd s - \int_0^\tau \int_{\R^2} f_n \psi \dd x \dd s
    \end{align}
    By passing to the limit $n \to \infty$, we deduce that, for every $\phi \in C^\infty((0,\tau] \times \R^2)$,
    \begin{align}\label{eq:weakformphi}
        \int_0^\tau \int_{\R^2} \widetilde{\rho} \partial_s \phi + \widetilde{\rho} u \cdot \nabla \phi \dd x \dd s = \int_{\R^2} \widetilde{\rho}(\tau) \phi(\tau) \dd x + \int_0^\tau \int_{\R^2} f \phi \dd x \dd s
    \end{align}
    and that, for every $s \in (0,\tau]$ and every $\psi \in C^\infty_c(\R^2),$
    \begin{align}\label{eq:weakformpsi}
        \int_{\R^2} \widetilde{\rho}(\tau) \psi \dd x - \int_{\R^2} \widetilde{\rho}(s) \psi \dd x = \int_0^\tau \int_{\R^2} \widetilde{\rho} u \cdot \nabla \psi \dd x \dd s - \int_0^\tau \int_{\R^2} f \psi \dd x \dd s
    \end{align}
    From \cref{eq:weakformphi}, we see that $\widetilde{\rho}$ is a weak solution of \cref{eq:backwardtransportproof}, and by combining \cref{eq:weakformpsi} with the fact that $\widetilde{\rho} \in L^\infty((0,\tau) \times \R^2)$ we deduce that 
    \begin{align*}
        \widetilde{\rho} \in C_{w^*} ((0,\tau];L^\infty(\R^2)).
    \end{align*}
    Now we have that in the sense of distributions
    \begin{align*}
        \begin{cases}
            \partial (\rho - \widetilde{\rho}) + \div((\rho - \widetilde{\rho}) u) = 0, \quad \tau<s \\
            \rho(s) - \widetilde{\rho}(s) = 0.
        \end{cases}
    \end{align*}
    With \cref{prop:wellposedtrans}, we conclude that, $\widetilde{\rho} = \rho$, \ie, for every $s \in (0,\tau]$ and almost every $x \in \R^2,$
    \begin{align*}
        \rho(s,x) = \widetilde{\rho}(s,x) = \rho(\tau,X(\tau,s,x)) + \int_\tau^s f(z,X(s,z,x)) \dd z,
    \end{align*}
    which completes the proof.
\end{proof}

\end{appendix}

\section*{Acknowledgments}
The author is supported by the Deutsche Forschungsgemeinschaft (DFG) via the project ``Inhomogeneous and compressible fluids: statistical solutions and dissipative anomalies'' within the SPP 2410 ``Hyperbolic Balance Laws in Fluid Mechanics: Complexity, Scales, Randomness'' (CoScaRa).

Part of this work was done while the author was visiting the Institut de Mathématiques de Toulouse. The author is grateful the kind hospitality of its members. The visit was supported by CIMI Labex.

The author is indebted to the anonymous referees for their valuable suggestions and comments on the manuscript. The author thanks Timothée Crin-Barat, Gianluca Crippa, Raphaël Danchin and Emil Wiedemann for some helpful conversations.

\medbreak
\noindent \textbf{Data availability statement} \,
 Data sharing not applicable to this article as no data sets were generated or analyzed during the current study.
 
% \bigbreak
% \noindent \textbf{Declarations}
%  \medbreak
\medbreak
\noindent \textbf{Conflicts of interest} The authors have no competing interests to declare that are relevant to the content of this
article.

\vspace{3mm}
\bibliographystyle{abbrv}
\bibliography{Ns-ref.bib}

\vfill 

\end{document}